\documentclass[11pt,a4paper]{amsart}
\usepackage{amsmath,amsthm,amssymb}
\usepackage{comment}
\newcommand{\mychoice}[3]{#1
}
\mychoice{
\newcommand{\plabel}[1]{ \label{#1}}
\newcommand{\gbibitem}[1]{ \bibitem{#1}}
\newcommand{\snewpage}{}

\specialcomment{commentx}{}{}\excludecomment{commentx}
\specialcomment{commenty}{}{}\excludecomment{commenty}
\specialcomment{commentz}{}{}

}{
\usepackage{xcolor}
\newcommand{\plabel}[1]{ \label{#1}\rlap{\smash{${}^{^{[#1]}}$}}}
\newcommand{\gbibitem}[1]{ \bibitem{#1}\rlap{\smash{${}^{^{[#1]}}$}}}
\newcommand{\snewpage}{\newpage}

\newenvironment{commentx}{\color{magenta} }{\color{black} }
\newenvironment{commenty}{\color{blue} }{\color{black} }

}{
\usepackage{xcolor}
\newcommand{\plabel}[1]{ \label{#1}}
\newcommand{\gbibitem}[1]{ \bibitem{#1}}
\newcommand{\snewpage}{}

\specialcomment{commentx}{}{}\excludecomment{commentx}

\specialcomment{commentz}{}{}\excludecomment{commentz}

}
\DeclareMathOperator{\nonass}{n-a}

\DeclareMathOperator{\Lie}{Lie}
\DeclareMathOperator{\cmr}{cmr}
\DeclareMathOperator{\ad}{ad}
\DeclareMathOperator{\BCH}{BCH}
\DeclareMathOperator{\Tr}{Tr}
\DeclareMathOperator{\tr}{tr}
\DeclareMathOperator{\cw}{cw}
\DeclareMathOperator{\rem}{rem}

\DeclareMathOperator{\degg}{\mathfrak{deg}}
\DeclareMathOperator{\Gal}{Gal}
\DeclareMathOperator{\Id}{Id}
\DeclareMathOperator{\id}{id}

\newcommand{\Starting}{\mathrm E^{\mathrm L}}
\newcommand{\Ending}{\mathrm E^{\mathrm R}}
\newcommand{\bo}{\boldsymbol}

\newcommand{\botimes}{{\textstyle\bigotimes}}

\newcommand{\leaveout}[1]{}
\newcommand{\ass}{\mathrm{assoc}}
\newcommand{\qedexer}{\renewcommand{\qedsymbol}{$\diamondsuit$}\qed\renewcommand{\qedsymbol}{$\Box$}}
\newcommand{\qedremark}{  \renewcommand{\qedsymbol}{$\triangle$} \qed \renewcommand{\qedsymbol}{$\Box$}}
\newcommand{\qedno}{\renewcommand{\qedsymbol}{}}
\newcommand{\eqed}{
\pushQED{\qed}
\qedhere
\popQED
}
\theoremstyle{definition}
\newtheorem{point}{}[section]
\newtheorem{remark}[point]{Remark}

\newtheorem{example}[point]{Example}
\theoremstyle{plain}
\newtheorem{prop}[point]{Proposition}
\newtheorem{lemma}[point]{Lemma}
\newtheorem{theorem}[point]{Theorem}
\newtheorem{cor}[point]{Corollary}
\newcommand{\marginextend}[1]{ \addtolength{\oddsidemargin}{-#1}  \addtolength{\evensidemargin}{-#1}\addtolength{\textwidth}{#1}\addtolength{\textwidth}{#1}}
\newcommand{\updownextend}[1]{ \addtolength{\topmargin}{-#1}  \addtolength{\textheight}{#1}
\addtolength{\textheight}{#1}}
\marginextend{1.25cm}
\updownextend{0cm}
\allowdisplaybreaks[3]
\begin{document}
\title{The Dynkin--Specht--Wever lemma and some associated constructions}
\date{\today}
\author{Gyula Lakos}
\email{lakos@renyi.hu}
\address{Alfréd Rényi Institute of Mathematics}
\keywords{Dynkin--Specht--Wever lemma, free Lie algebras, Kashiwara--Vergne problem}
\subjclass[2010]{Primary: 17B35. Secondary: 17B01, 16S30.}
\begin{abstract}
In the first part of the paper, some extensions of the classical Dynkin--Specht--Wever lemma are developed.
In the second part, we extend Burgunder's splitting construction, and relate back to the
Kashiwara--Vergne problem.
\end{abstract}
\maketitle

\section*{Introduction}\plabel{sec:intro}
\textbf{The setup of the Dynkin--Specht--Wever lemma.}
Free Lie $K$-algebras $\mathrm F_K^{\Lie}[X_\lambda :\lambda\in\Lambda]$ can be considered over any coefficient ring $K$.
The free associate algebra $\mathrm F_K[X_\lambda :\lambda\in\Lambda]$ is also Lie $K$-algebra with respect to
 associative commutators, thus there is natural commutator evaluation map
 $\iota:\mathrm F_K^{\Lie}[X_\lambda :\lambda\in\Lambda]\rightarrow\mathrm F_K[X_\lambda :\lambda\in\Lambda]$.
In a more abstract viewpoint, $\mathrm F_K[X_\lambda :\lambda\in\Lambda]$ naturally isomorphic to universal
 enveloping algebra $\mathcal U\mathrm F_K^{\Lie}[X_\lambda :\lambda\in\Lambda]$, and then $\iota$ is the natural
 evaluation map into the universal enveloping algebra.
It is the Magnus--Witt theorem (see Magnus \cite{MM}, Witt \cite{W}) that $\iota$ is injective.
This can be thought as the weak (filtration level $1$) version of the Poincaré--Birkhoff--Witt theorem (cf. Birkhoff \cite{B}, Witt \cite{W})
 applied to the free Lie $K$-algebra $\mathrm F_K^{\Lie}[X_\lambda :\lambda\in\Lambda]$
 (except the applicability of the PBW theorem is not completely trivial unless $K$ is a field);
 otherwise it is also a simple consequence of the eliminiation
 technique of Magnus (see Magnus \cite{MM} or Magnus, Karrass, Solitar \cite{MKS}) or the one of
 Shirshov and Lazard (see \v{S}ir\v{s}ov \cite{S0}, Lazard \cite{L2},
 Bourbaki \cite{BX}).
The Dynkin--Specht--Wever lemma of Dynkin \cite{Dy}, Specht \cite{Sp}, Wever \cite{We},
 helps to reconstruct $P^{\Lie}\in \mathrm F_K^{\Lie}[X_\lambda :\lambda\in\Lambda]$
 (imagined as a linear combination of Lie monomials) from $P^{\ass}=\iota(P^{\Lie})\in \mathrm F_K [X_\lambda :\lambda\in\Lambda]$
 (imagined as a linear combination of associative monomials) using simple iterated commutators.
This tool works most properly if $\mathbb Q\subset K$ (and, in particular, it yields the Magnus--Witt theorem).
Otherwise, it might yield only partial information, and more sophisticated techniques involving
 free Lie algebra bases might be needed (cf. Reutenauer \cite{R}).
If we assume  $\mathbb Q\subset K$,
 then the Dynkin--Specht--Wever lemma can also be thought as giving a projection
 in  $\mathrm F_K^{\Lie}[X_\lambda :\lambda\in\Lambda]$ to
 $\mathrm F_K^{\cmr}[X_\lambda :\lambda\in\Lambda]\equiv\iota(\mathrm F_K^{\Lie}[X_\lambda :\lambda\in\Lambda])$,
 i. e. a ``Lie idempotent'', and concretely
 the ``Dynkin idempotent'' (but there are `left' and `right' versions).
Considering the kernel of the projection, it is not quite as nice as the first canonical projection or
 the ``Eulerian idempotent'' (cf. Solomon \cite{SS}, Reutenauer \cite{R}).
However, what makes the Dynkin method successful is its simplicity.
Therefore it might be interesting to discuss variants of the lemma or some associated constructions
 which retain that (or close) practical level of explicitness.

\textbf{In the first part of this paper,} we investigate some extensions of the classical Dynkin--Specht--Wever lemma.
We start with the weighted DSW presentation.
Multigrade-wise,  this can be thought as the (infinite) linear combination of principal DSW presentations.
Then a centrally bracketed version of the Dynkin--Specht--Wever lemma is exhibited.
Next a (naturally weak) generalization to the universal enveloping algebra is considered.
In this first part, the Lie algebraic and associative commutator algebraic levels are kept
 carefully separated, and in the beginning $\mathbb Q\subset K$ is not assumed.

\textbf{In the second part of this paper,} the character of the discussion changes.
We   assume that $\mathbb Q\subset K$ and the natural identifications
 between free Lie algebras and the associative commutator algebras will be taken granted.
Here we discuss the Dynkin--Burgunder splitting maps.
We give a description of their behaviour on Lie monomials and we
 comment on their relationship to the Kashiwara--Vergne conjecture of
 Kashiwara, Vergne \cite{KV}, which is by now the Alexeev--Meinrenken--Torossian theorem,
 see Alexeev, Meinrenken \cite{AM2},  Alexeev, Torossian \cite{AT};
 which was the original movation for Burgunder's construction.

\textbf{The setup of the Dynkin--Burgunder splitting maps.}
Let us consider the free Lie algebra $\mathrm F^{\Lie}[X_\lambda : \lambda\in\Lambda]$
 (over a coefficient ring $K\supset \mathbb Q$, therefore omitted).
We identify $\mathrm F^{\Lie}[X_\lambda : \lambda\in\Lambda]$ with the commutator Lie subalgebra of $\mathrm F[X_\lambda : \lambda\in\Lambda]$
 generated by the $X_\lambda$.
The algebras $\mathrm F^{\Lie}[X_\lambda : \lambda\in\Lambda]$ and $\mathrm F[X_\lambda : \lambda\in\Lambda]$
 are naturally and compatibly graded.
Let $\mathrm F^{\Lie}[X_\lambda : \lambda\in\Lambda]_{(n)}$ denote the $n$-homogeneous part of
 $\mathrm F^{\Lie}[X_\lambda : \lambda\in\Lambda]$.
If $n\geq2$, then $\mathrm F^{\Lie}[X_\lambda : \lambda\in\Lambda]_{(n)}$ is a sum a commutators,
 which can expanded to a sum  of commutators $X_\lambda$ and elements from  $\mathrm F^{\Lie}[X_\lambda : \lambda\in\Lambda]_{(n-1)}$.
In certain situations it may be beneficial to consider a natural map
\begin{equation}
\delta^{\mathrm L}_n: \mathrm F^{\Lie}[X_\lambda : \lambda\in\Lambda]_{(n)}\rightarrow
\mathrm F^{\Lie}[X_\lambda : \lambda\in\Lambda]_{(1)}\otimes \mathrm F^{\Lie}[X_\lambda : \lambda\in\Lambda]_{(n-1)}
\plabel{eq:splita}
\end{equation}
 such that
\begin{equation}
[\cdot,\cdot]_{(1,n-1)}\circ \delta^{\mathrm L}_n=\id_{\mathrm F^{\Lie}[X_\lambda : \lambda\in\Lambda]_{(n)}};
\plabel{eq:basid}
\end{equation}
 where $[\cdot,\cdot]_{(1,n-1)}$ is just the restriction of the commutator map to the corresponding subspace.
Using the Dynkin--Specht--Wever lemma, such a map can be constructed easily.
This was done by Burgunder \cite{Bur}.

As we have mentioned, the original motivation for that was the Kashiwara--Vergne conjecture.
Let us recall this.
As it is well-known, $\BCH(X,Y)\equiv\log((\exp X)(\exp Y))$ is a formal Lie series in $X,Y$
which can also be generated relatively explicitly, and which is analytic for $X,Y\sim 0$ in appropriate sense.
This is the Baker--Campbell--Hausdorff theorem, see Bonfiglioli, Fulci \cite{BF}, Achilles, Bonfiglioli \cite{AB}.
The Kashiwara--Vergne conjecture, in a universal version, postulates the existence of analytic Lie series $F(X,Y)$ and $G(X,Y)$
such that

(a) $X+Y-\BCH(Y,X)=[X,\alpha(-\ad X)F(X,Y)]+[Y,\alpha(\ad Y) G(X,Y)]$,

(b) $\Tr\left( \Ending_X  F(X,Y)+\Ending_Y  G(X,Y) \right)=
\frac12\Tr\left( \beta(X)+\beta(Y)-\beta(\BCH(X,Y))-1 \right)$,\\
 where
 $\alpha(u)=\frac{\mathrm e^{u}-1}{u}$ and $\beta(u)=\frac{u}{\mathrm e^{u}-1}$ are formal power series;
 $\ad S: W\mapsto[S,W]$ is the well-know adjoint action; $\Tr$ is the formal trace
 sending associative monomials to cyclically symmetrized monomials;
 $\Ending_X$ and $\Ending_Y$ select the $X$-ending and $Y$-ending associative monomials, respectively.

Burgunder's construction was devised to treat (a) systematically.
Ultimately, it is not a particularly effective way to deal with the Kashiwara--Vergne problem,
 as it is not particularly specific to that question.
However, we will argue that it is yet specific in some partial ways:
Similarly to the Dynkin--Specht--Wever lemma, Burgunder's construction
 allows principal versions, which solve the easier (odd) part of the Kashiwara--Vergne problem.

\textbf{Acknowledgements.}
 The author is thankful for the hospitality of the Alfréd Rényi Isntitute of Mathematics.
 The author is grateful to Christophe Reutenauer.

\snewpage
\section{The weighted Dynkin--Specht--Wever lemma}\plabel{sec:dsw}

\textbf{On free Lie algebras.}
Free Lie $K$-algebras (or any other kinds of free algebras) do not require particularly specific constructions.
But, it  is  useful to have  theorems providing  control over them.
The most basic and important observation about the free Lie $K$-algebra $\mathrm F_K^{\Lie}[X_\lambda :\lambda\in\Lambda]$
 is that it is multigraded by multiplicity of variables $X_\lambda$.
The corresponding statement is trivial for the free nonassociative $K$-algebra (magma, brace algebra) $\mathrm F^{\nonass}_K[X_\lambda :\lambda\in\Lambda]$,
and also for the free associative $K$-algebra   $\mathrm F_K[X_\lambda :\lambda\in\Lambda]$.
Now, the free Lie $K$-algebra $\mathrm F_K^{\Lie}[X_\lambda :\lambda\in\Lambda]$ can be thought
$\mathrm F^{\nonass}_K[X_\lambda :\lambda\in\Lambda]$
factorized further by the ideal $\mathrm I_K^{\Lie}[X_\lambda :\lambda\in\Lambda]$
generated by all elements $\bo[Y_1,Y_1\bo]$ and
$\bo[\bo[Y_1,Y_2\bo],Y_3\bo]+\bo[\bo[Y_2,Y_3\bo],Y_1\bo]+\bo[\bo[Y_3,Y_1\bo],Y_2\bo]$
such that $Y_1,Y_2,Y_3\in \mathrm F^{\nonass}_K[X_\lambda :\lambda\in\Lambda]$.
At first sight the homogeneity properties might not be clear, however, the generators expanded,
we see that $\mathrm I_K^{\Lie}[X_\lambda :\lambda\in\Lambda]$ is generated by all elements
$\bo[Z_1,Z_1\bo]$ and $\bo[Z_1,Z_2\bo]+\bo[Z_2,Z_1\bo]$ and
$\bo[\bo[Z_1,Z_2\bo],Z_3\bo]+\bo[\bo[Z_2,Z_3\bo],Z_1\bo]+\bo[\bo[Z_3,Z_1\bo],Z_2\bo]$
where $Z_1,Z_2,Z_3\in \mathrm F^{\nonass}_K[X_\lambda :\lambda\in\Lambda]$ are monomonials.
Thus $\mathrm I_K^{\Lie}[X_\lambda :\lambda\in\Lambda]$ is a homogeneously generated and therefore homogeneous
ideal (with respect to the multigrading),
making   $\mathrm F_K^{\Lie}[X_\lambda :\lambda\in\Lambda]$ multigraded.

\textbf{The  Dynkin--Specht--Wever lemma.}
This statement is a  consequence of the multigradedness of the free Lie $K$-algebras.
We present the weighted version.
Suppose that we assign a weight $w_\lambda\in \mathbb Z$  or $w_\lambda\in K$ to every variable $X_\lambda$.
Let $w: \mathrm F^{\Lie}_K[X_\lambda :\lambda\in\Lambda]\rightarrow \mathrm F^{\Lie}_K[X_\lambda :\lambda\in\Lambda]$  be the map which multiplies
by ${m_1}w_{\lambda_1}+\ldots+ {m_s}w_{\lambda_s}$ in multigrade $X_{\lambda_1}^{m_1}\ldots X_{\lambda_s}^{m_s}$.
Also note that commutator evaluation into the free associative algebra $\mathrm F_K[X_\lambda :\lambda\in\Lambda]$
 is compatible to multigrading.
For practical reasons we will use left-iterated higher commutators
$\bo[X_1,\ldots, X_n\bo]_{\mathrm L}=\bo[ X_1,\bo[X_2,\ldots, \bo[X_{n-1},X_n\bo]\ldots\bo]\bo]$.
\begin{theorem}\plabel{lem:DSW}(Weighted Dynkin--Specht--Wever lemma.) Suppose that
$P^{\Lie}(X_1,\ldots,X_n)\in\mathrm F^{\Lie}[ X_1,\ldots,X_n]$  expands in the commutator-evaluation to the noncommutative polynomial
\[P^{\ass}(X_1,\ldots,X_n)=\sum_s a_s X_{i_{s,1}}\ldots X_{i_{s,p_s}}. \]
Then
\[w(P^{\Lie}(X_1,\ldots,X_n))=\sum_s a_s \bo[X_{i_{s,1}}\ldots ,X_{i_{s,p_s-1}} ,w(X_{i_{s,p_s}})\bo]_{\mathrm L}. \]
\begin{proof}[Proof using derivations]
Consider $\mathrm F^{\Lie}_K[X_1,\ldots,X_n]\oplus K\mathbf u$, and extend the Lie bracket such
 that $\bo[\mathbf u,\mathbf u\bo]=0$, and $\bo[Q,\mathbf u\bo]=w(Q)$, $\bo[\mathbf u,Q\bo]=-w(Q)$.
This yields a Lie $K$-algebra; it is sufficient to check $[x,x]=0$, $[x,y]+[y,x]=0$, $[[x,y],z]+[[y,z],x]+[[z,x],y]=0$
 when $x,y,z$ are Lie-monomials or $\mathbf u$.
(This is taking the semidirect product with respect to the grade derivation.)
Then
\begin{align}\notag
w(P^{\Lie}(X_1,\ldots,X_n))&=\bo[P^{\Lie}(X_1,\ldots,X_n),\mathbf u\bo]
\\\notag&=(\ad P^{\Lie}(X_1,\ldots,X_n)  )\mathbf u\notag
\\\notag&=P^{\ass}(\ad X_1,\ldots,\ad X_n)\mathbf u
\\&=\sum_s a_s \bo[X_{i_{s,1}}\ldots ,X_{i_{s,p_s-1}} ,X_{i_{s,p_s}},\mathbf u\bo]_{\mathrm L}\notag
\\&=\sum_s a_s \bo[X_{i_{s,1}}\ldots ,X_{i_{s,p_s-1}} ,w(X_{i_{s,p_s}})\bo]_{\mathrm L}.\qquad\qedhere
\notag
\end{align}
\end{proof}
\end{theorem}
\begin{proof}[Proof using noncommutative polynomials.]
First we prove the statement when $P^{\Lie}$ is homogeneous in the multigrading
such that every variable has multiplicity at most $1$.
We can actually assume that every variable $X_1,\ldots,X_n$ is used exactly once.
Then we can write
\begin{equation}
P^{\ass}(X_1,\ldots,X_n)=\sum_{\sigma\in\Sigma_n}a_\sigma X_{\sigma(1)}\ldots a_\sigma X_{\sigma(n)}.\plabel{eq:EE}
\end{equation}

Let us fix an arbitrary element $k\in\{1,\ldots,n\}$.
Now, $P^{\Lie}(X_1,\ldots,X_n)$ is a Lie-polynomial, thus, using standard commutator rules,
we can write it as linear combination of terms $\bo[X_{i_1},\ldots,X_{i_{n-1}} ,X_k\bo]_{\mathrm L}$,
where $\{i_1,\ldots,i_{n-1}\}=\{1,\ldots,n\}\setminus\{k\}$. However, evaluated in the
 noncommutative polynomial algebra, such a commutator expression gives
only one monomial contribution $X_{i_1}\ldots X_{i_{n-1}} X_k $ such that the last term is $X_k$.
Thus, the coefficient of $\bo[X_{i_1},\ldots,X_{i_{n-1}} ,X_k\bo]_{\mathrm L}$ can be read off from \eqref{eq:EE}.
We find that
\begin{equation}
P^{\Lie}(X_1,\ldots,X_n)=
\sum_{\sigma\in\Sigma_n, \sigma(n)=k}a_\sigma \bo[X_{\sigma(1)},\ldots ,X_{\sigma(n-1)},X_k\bo]_{\mathrm L}.\plabel{eq:EEk}
\end{equation}
(This is the noncommutative comparison technique of Dynkin.)
Multiplying \eqref{eq:EEk} by the weight $w_k$, and adding these for all $k\in\{1,\ldots,n\}$, we obtain the statement.

In the general case, $P^{\Lie}$ is a linear combination of monomials of $X_1,\ldots,X_n$.
In those monomials we can change the variables $X_i$ by some other ones from $X_{i,1},\ldots, X_{i,f_i}$
erratically such that in the resulted $\tilde P^{\Lie}$ every variable has multiplicity at most $1$.
For the new variables  $X_{i,j}$ we define the weight $w_{i,j}$ as $w_i$.
Now we can apply statement for $\tilde P^{\Lie}$, and then we can simply forget the second indices in the variables $X_{i,j}$.
(This is the technique of non-deterministic polarization.)
\end{proof}
Similar statements hold with respect to the right-iterated Lie-com\-mu\-ta\-tors
 $\bo[X_1,\ldots, X_n\bo]_{\mathrm R}$ $=\bo[\bo[\ldots\bo[ X_1,X_2\bo],\ldots,X_{n-1}\bo],X_n\bo]$.
The ``unweighted''  Dynkin--Specht--Wever lemma is when every weight  $w_i$ is $1$, and
the weight map is $w=\degg$, multiplcation by the joint degree of variables.
If $\mathbb Q\subset K$, then, in the unweighted case,  $w^{-1}=\degg^{-1}$ can be applied, showing that
$P^{\ass}$ determines $P^{\Lie}$, which is the statement of the Magnus--Witt theorem.
After the unweighted case, the most important case is when only one variable has nonzero weight.
We spell out: If $P^{\Lie}$ (which is as earlier) is $h_k$-homogeneous in the variable $X_k$, then
\[h_k\cdot P^{\Lie}(X_1,\ldots,X_n)=\sum_{\substack{s \\ i_{s,p_s}=k}}  a_s\bo[X_{i_{s,1}}\ldots ,X_{i_{s,p_s-1}} , X_{i_{s,p_s}} \bo]_{\mathrm L}. \]
Note that the Dynkin--Specht--Wever lemma retains interest even in the associative setting.
Abstractly this is when  $P^{\ass}\in \mathrm F^{\cmr}_K[X_1,\ldots,X_n]$ is
 given and we look for $P^{\nonass}\in \mathrm F^{\nonass}_K[X_1,\ldots,X_n]$
 such that $P^{\nonass}$ commutator evaluates to $P^{\ass}$.
\begin{example}
The following example illustrates the weighted Dynkin--Specht--Wever lemma in the
original setting of Dynkin \cite{Dy}.

As the components of $\log(\exp( X)  \exp(  Y))$ are commutator polynomials
(which can also be shown in several ways),
we can apply the Dynkin--Specht--Wever lemma  to (the homogeneous parts) of the power series
expansion
\begin{align}\log((\exp X)(\log Y))=&\log\left(1+\sum_{\substack{p,q\geq 0\\p+q\geq1}}\frac{X^pY^q}{p!\,q!}  \right)
\plabel{eq:eppe}\\=&\sum_{k=1}^\infty\frac{(-1)^{k-1}}{k}\sum_{
\substack{p_1\ldots,p_k,q_1\ldots,q_k\geq0\\ p_1+q_1,\ldots,p_k+q_k\geq 1}}
\frac{X^{p_1}Y^{q_1}\ldots X^{p_k}Y^{q_k}}{p_1!\ldots p_k!\,q_1!\ldots q_k!}.
\notag
\end{align}
In this standard manner, commas in $[]_{\mathrm L}$ omitted,  we obtain the formula of Dynkin \cite{Dy},
\begin{multline}
\log((\exp X)(\exp Y))=\\=\sum_{k=1}^{\infty}\frac{(-1)^{k-1}}{k}\sum_{
\substack{p_1\ldots,p_k,q_1\ldots,q_k\geq0\\ p_1+q_1,\ldots,p_k+q_k\geq 1}}
\frac{[X^{p_1}Y^{q_1}\ldots X^{p_k}Y^{q_k}]_{\mathrm L}}{(p_1+\ldots+p_k+q_1+\ldots+q_k)p_1!\ldots p_k!\,q_1!\ldots q_k!}.
\plabel{eq:eps}\end{multline}

Some works, e.~g.~Kol\'a\v{r}, Michor, Slov\'ak \cite{KMS}, or
Duistermaat, Kolk \cite{DK}
present
\begin{multline}
\log((\exp X)(\exp Y))=\\=Y+\sum_{k=0}^{\infty}\frac{(-1)^{k}}{k+1}\sum_{
\substack{p_1\ldots,p_k,q_1\ldots,q_k\geq0\\ p_1+q_1,\ldots,p_k+q_k\geq 1}}
\frac{[X^{p_1}Y^{q_1}\ldots X^{p_k}Y^{q_k}X]_{\mathrm L}}{(p_1+\ldots+p_k+1)p_1!\ldots p_k!\,q_1!\ldots q_k!}
\plabel{eq:epps}\end{multline}
 as the BCH formula/ ``Dynkin's formula'', which they prove by
 differential equational/ geometric means, but essentially just by using the old Schur(--Poincar\'e)  formula
\[\frac{\mathrm d \log(\exp(tX)\exp(Y))}{\mathrm dt}=\beta(-(\exp \ad t X)(\exp \ad Y))X;\]
i.~e.~by very classical methods.
(Cf.~ Schur \cite{Sch}, \cite{Sch2}, Poincar\'e \cite{P},
Duistermaat \cite{Du}, Achilles, Bonfigioli \cite{AB}, Bonfiglioli, Fulci \cite{BF}.)

Now, \eqref{eq:epps} can also be realized algebraically from \eqref{eq:eppe} but by applying the weighted Dynkin--Specht--Wever lemma
with weight prescription $\deg_X X=1$, $\deg_X Y=0$:
The part when the total weight is $0$ can be seen to be  $Y$ easily.
(We cannot apply $(\degg_X)^{-1}$ on that part, therefore it is dealt separately.)
Then  relabel $k$ to $j+1$, and  notice that  only the $q_{j+1}=0$, $p_{j+1}=1$ part survives
weighting and commutatoring, respectively.

One can also apply the weight prescription $\deg_Y X=0,\deg_Y Y=1$. Another possibility is
to apply $\log((\exp X)(\log Y))=-\log((\exp -Y)(\exp -X))$, which also corresponds to the rewriting of the $[]_{\mathrm R}$-version
to $[]_{\mathrm L}$-terminology.
Thus we have a couple of formulas of Dynkin type.
(The argument also works in the other direction:
As the $\deg_X$ and $\deg_Y$ weightings can be obtained by classical methods,
 taking appropriate convex combinations multigradewise, we obtain Dynkin's actual formula \eqref{eq:eps}.)
\qedexer
\end{example}
Regarding the proof of Theorem \ref{lem:DSW}, this is a situation where an abstract approach
 using derivations and a more classical approach using non-commutative polynomials are equally available.
Both approaches have their advantages:

In terms of derivations, we might have the situation that
\[P^{\ass}(X_1,\ldots,X_n)=\sum_s a_s P^{\ass}_{{s,1}}\ldots P^{\ass}_{{s,p_s}}, \]
where $P^{\ass}_{{s,j}}$ is the commutator evaluation of  $P^{\Lie}_{{s,j}}$.
Then the method of derivations quickly tells that
\[w(P^{\Lie}(X_1,\ldots,X_n))=\sum_s a_s \bo[P^{\Lie}_{{s,1}}\ldots ,P^{\Lie}_{{s,p_s-1}} ,w(P^{\Lie}_{{s,p_s}})\bo]_{\mathrm L}. \]

In terms of the polynomial method, it is used to obtain further information by
Dzhumadil’daev \cite{Dzh}, where the weighted Dynkin--Specht--Wever lemma is considered in the
multilinear (multiplicity-free) case.
Another example for its application is the following
\section{Centrally bracketed Dynkin--Specht--Wever lemma}
For the sake of simplicity, we present an unweighted version.
\begin{theorem}\plabel{th:CDSW} Suppose that
$P^{\Lie}(X_1,\ldots,X_n)\in\mathrm F_K^{\Lie}[X_1,\ldots,X_n]$  is of homogeneity degree $h$ (in its variables jointly),
and it expands in the commutator evaluation to the noncommutative polynomial
\[P^{\ass}(X_1,\ldots,X_n)=\sum_s a_s X_{i_{s,1}}\ldots X_{i_{s,h}}. \]
Then
\[h(h-1)\cdot P^{\Lie}(X_1,\ldots,X_n)=\sum_s \sum_{p=1}^{n-1}
a_s\bo[\bo[X_{i_{s,1},},\ldots ,X_{i_{s,p}}\bo]_{\mathrm L}, \bo[X_{i_{s,p+1}} ,\ldots,X_{i_{s,h}} \bo]_{\mathrm R}\bo]. \]
\begin{proof}
Again, first we prove the statement when $P^{\Lie}$ is homogeneous in the multigrading
such that every variable has multiplicity at most $1$.
We actually assume that every variable $X_1,\ldots,X_n$ is used exactly once. Then
\begin{equation}
P^{\ass}(X_1,\ldots,X_n)=\sum_{\sigma\in\Sigma_n}a_\sigma X_{\sigma(1)}\ldots a_\sigma X_{\sigma(n)}.\plabel{eq:EEE}
\end{equation}
Let us fix  $k\neq l\in\{1,\ldots,n\}$.
Using standard commutator rules,
 $P^{\Lie}(X_1,\ldots,X_n)$,
 can be written  as a linear combination of terms
 $\bo[\bo[X_{i_1},\ldots,X_{i_{p}} ,X_k\bo]_{\mathrm L}, \bo[X_l ,X_{i_{p+1}},\ldots,X_{i_{n-2}}\bo]_{\mathrm R}\bo]$,
 where $\{i_1,\ldots,i_{n-2}\}=\{1,\ldots,n\}\setminus\{k,l\}$.
However, evaluated in the noncommutative polynomial algebra, such a commutator expression gives
 only one monomial contribution $X_{i_1}\ldots X_{i_p}X_k X_lX_{i_{p+1}} \ldots X_{i_{n-2}} $
 such that $X_k$ is immediately followed by $X_l$.
Compared this to \eqref{eq:EEE}, we  find that
\begin{equation}
P^{\Lie}(X_1,\ldots,X_n)
=\sum_{\substack{\sigma\in\Sigma_n,\\ \sigma(p)=k,\\ \sigma(p+1)=l}}a_\sigma
\bo[\bo[X_{\sigma(1)},\ldots,X_{\sigma(p-1)} ,X_k\bo]_{\mathrm L},
\bo[X_l ,X_{\sigma(p+2)},\ldots,X_{\sigma(n)} \bo]_{\mathrm R}\bo].
\plabel{eq:FFk}
\end{equation}
Summing this for all possible pairs $k,l$, we obtain
\begin{equation}
n(n-1)\cdot P^{\Lie}_n(X_1,\ldots,X_n)
=\sum_{\substack{\sigma\in\Sigma_n,\\ 1\leq p<n}}a_\sigma
\bo[\bo[X_{\sigma(1)},\ldots ,X_{\sigma(p)}\bo]_{\mathrm L}, \bo[X_{\sigma(p+1)} ,\ldots,X_{\sigma(n)}
\bo]_{\mathrm R}\bo].
\plabel{eq:FFn}
\end{equation}
This proves the statement in the special case.
The general case follows from considering a nondeterministic polarization as before.
\end{proof}
\end{theorem}
Weighted versions are also possible, but they require a sort of quadrating grading.
Therefore we formulate only special cases.
Suppose that
$P^{\Lie}(X_1,\ldots,X_n)\in\mathrm F^{\Lie}[ X_1,\ldots,X_n]$ expands in the commutator evaluation to the noncommutative polynomial
\[P^{\ass}(X_1,\ldots,X_n)=\sum_s a_s X_{i_{s,1}}\ldots X_{i_{s,p_s}}. \]

Assume that $k\neq l\in\{1,\ldots,n\}$ and $P^{\Lie}$ is $h_k$-homogeneous in $X_k$ and $h_l$-homogeneous in $X_l$.
Then
\[h_k\cdot h_l\cdot P^{\Lie}(X_1,\ldots,X_n)=\sum_s \sum_{\substack{1\leq r<p_s\\i_{s,r}=k,\,i_{s,r+1}=l }}
a_s\bo[\bo[X_{i_{s,1},},\ldots ,X_{i_{s,r}}\bo]_{\mathrm L}, \bo[X_{i_{s,r+1}} ,\ldots,X_{i_{s,p_s}} \bo]_{\mathrm R}\bo]. \]

Assume that $k\in\{1,\ldots,n\}$ and $P^{\Lie}$ is $h_k$-homogeneous in $X_k$ .
Then
\[h_k (h_k-1)\cdot P^{\Lie}(X_1,\ldots,X_n)=\sum_s \sum_{\substack{1\leq r<p_s\\i_{s,r}=i_{s,r+1}=k }}
a_s\bo[\bo[X_{i_{s,1},},\ldots ,X_{i_{s,r}}\bo]_{\mathrm L}, \bo[X_{i_{s,r+1}} ,\ldots,X_{i_{s,p_s}} \bo]_{\mathrm R}\bo]. \]
These latter statements, in general, show that commutators are mixing the variables well.

\snewpage
\section{The Dynkin--Specht--Wever lemma in the universal enveloping algebra}
\plabel{sec:dyn}
Originally, the DSW lemma is to be applied in the free setting.
It can be applied in the universal enveloping algebra $\mathcal U\mathfrak g$
 if that contains a large part of the noncommutative polynomial algebra.
This strong requirement can be relaxed a little bit if $\mathbb Q\subset K$.

\textbf{Reminder: the symmetric Poincaré--Birkhoff--Witt theorem.}
Recall that  if $\mathbb Q\subset K$, then the canonical homomorphism
$\boldsymbol m:\bigotimes\mathfrak g\rightarrow\mathcal U\mathfrak g$
is such that its restriction
\begin{equation*}\boldsymbol m_\Sigma:\textstyle{\bigotimes_\Sigma\mathfrak g}\rightarrow\mathcal U\mathfrak g\end{equation*}
to symmetrized tensor products (generated by $a_1\otimes_\Sigma \ldots\otimes_\Sigma a_n=\frac1{n!}\sum_{\sigma\in\Sigma_n}g_{\sigma(1)}\otimes\ldots\otimes g_{\sigma(n)}$) is an isomorphism.
This the symmetric version of the Poincaré--Birkhoff-Witt theorem, see Cohn \cite{CC}.
Thus we can assume that $\mathfrak g\subset \mathcal U\mathfrak g$.
The map $\mu_\Sigma=(\boldsymbol m_\Sigma)^{-1}\circ\boldsymbol m$ can be described quite explicitly:
There are multlinear rational Lie polynomials $\mu_n^{\Lie}(X_1,\ldots,X_n)$ $(n\geq1)$ such that
 $\mu_\Sigma:\bigotimes\mathfrak g\rightarrow  \bigotimes_\Sigma \mathfrak g$
is given by
\[
\mu_\Sigma(x_1\otimes\ldots\otimes x_n)=\sum_{\substack{I_1\dot\cup\ldots \dot\cup I_s=\{1,\ldots,n\}\\I_k=\{i_{k,1},\ldots,i_{k,p_k}\}\neq
\emptyset \\i_{k,1}<\ldots<i_{k,p_k}}}\frac1{s!}\,
 \mu^{\Lie}_{p_1}(x_{i_{1,1}},\ldots ,x_{i_{1,p_1}})\otimes\ldots\otimes \mu^{\Lie}_{p_s}(x_{i_{s,1}},\ldots ,x_{i_{s,p_s}})\notag
\]
(and it acts trivially in the $0$th order).
Due to its role, $\mu_n^{\Lie}(X_1,\ldots,X_n)$ is the so-called first canonical projection.
$\mu_n^{\Lie}(X_1,\ldots,X_n)$ expands in the commutator algebra as
\begin{equation}
 \mu_n^{\ass}(X_1,\ldots,X_n)=\sum_{\sigma\in\Sigma_n}\mu_\sigma X_{\sigma(1)}\ldots  X_{\sigma(n)}.\plabel{eq:FF}
\end{equation}
where the coefficients $\mu_\sigma$ can be given explicitly, see Solomon \cite{SS}, Reutenauer \cite{R}.
(cf. also Dynkin \cite{Dyy}, Magnus \cite{M}, Mielnik, Pleba\'nski \cite{MP}).
In terms of $\mathcal U\mathrm F^{\Lie}[X_\lambda:\lambda\in\Lambda]\simeq
\mathrm F[X_\lambda:\lambda\in\Lambda]$ the expression $\mu_n^{\ass}$ yields a ``Lie idempotent'' (in grade $n$),
which, due to the actual shape of the coefficients is also called as the Eulerian idempotent.

\textbf{DSW lemma in the universal enveloping algebra.}
By the (simplest layer of the) unweighted DSW lemma,
\begin{equation}
n\cdot\mu_n^{\Lie}(X_1,\ldots,X_n)=
\sum_{\sigma\in\Sigma_n}\mu_\sigma \bo[X_{\sigma(1)},\ldots ,X_{\sigma(n-1)},X_{\sigma(n)}\bo]_{\mathrm L}.\plabel{eq:EEn}
\end{equation}
\begin{prop}\plabel{prop:dsw1} ($\mathbb Q\subset K$)
Suppose that the $K$-submodule  $W\subset \botimes^n\mathfrak g$ is  closed for actions of  $\Sigma_n$
inducing  permutations in the order of tensor product.
Also assume that $\bo m|_{W} :W\rightarrow\mathcal U\mathfrak g$ is injective.
Now, if  \[P=\sum_{\lambda}a_{1,\lambda}\otimes\ldots\otimes a_{n,\lambda} \]
such that $P\in W$ and $\bo m(P)\in\mathfrak g$, then
\[n\cdot \bo m(P)=\sum_{\lambda}\,\bo[a_{1,\lambda},\ldots, a_{n,\lambda}\bo]_{\mathrm L}.\]
\begin{proof} $\bo m(P)\in\mathfrak g$ means that in $\mathcal U\mathfrak g$
\begin{equation}
\sum_{\lambda}a_{1,\lambda}\ldots a_{n,\lambda} =\bo m(P)=
\sum_{\lambda}\mu_n(a_{1,\lambda}\ldots, a_{n,\lambda})=
\sum_{\sigma\in\Sigma_n}\mu_\sigma\sum_{\lambda}a_{\sigma(1),\lambda}\ldots a_{\sigma(n),\lambda}.
\plabel{eq:tef}
\end{equation}
(By the PBW theorem,  in $\mathcal U \mathfrak g$, we are allowed to confuse $\mu^{\Lie}_n$ and $\mu^{\ass}_n$.)
Due to the injectivity of $\bo m|_{W}$, we find
\[\sum_{\lambda}a_{1,\lambda}\otimes\ldots \otimes a_{n,\lambda} =
\sum_{\sigma\in\Sigma_n}\mu_\sigma\sum_{\lambda}a_{\sigma(1),\lambda}\otimes \ldots \otimes a_{\sigma(n),\lambda}\]
(both sides are in $W$, because $W$ is permutation-invariant).
Then, applying the multilinear $\bo[,\bo]_{\mathrm L}$,  using \eqref{eq:EEn}, and, finally, \eqref{eq:tef}, we find
\begin{align*}
\sum_{\lambda}\bo[a_{1,\lambda},\ldots , a_{n,\lambda}\bo]_{\mathrm L} &=
\sum_{\sigma\in\Sigma_n}\mu_\sigma\sum_{\lambda}\bo[a_{\sigma(1),\lambda},
\ldots , a_{\sigma(n),\lambda}\bo]_{\mathrm L}=\\&=\sum_{\lambda}n\cdot\mu_n(a_{1,\lambda}, \ldots , a_{n,\lambda})
=n\cdot \bo m(P).\notag\qedhere
\end{align*}
\end{proof}
\end{prop}

The discussion extends to the weighted case.
If we assign the weight $w_i\in K$ to the variables $X_i$ (for accounting purposes), then one has
\begin{equation}
(w_1+\ldots+w_n)\cdot\mu^{\Lie}_n(X_1,\ldots,X_n)=
\sum_{\sigma\in\Sigma_n}\mu_\sigma \bo[X_{\sigma(1)},\ldots
,X_{\sigma(n-1)},w_{\sigma(n)} X_{\sigma(n)}\bo]_{\mathrm L}.
\plabel{eq:EEw}
\end{equation}

Assume that $\mathfrak g$ is $K$-graded  as a Lie $K$-algebra.
Then $\bigotimes\mathfrak g$ is also $K$-graded naturally.
Let $w: \bigotimes\mathfrak g\rightarrow\bigotimes\mathfrak g$ be the map which acts
as multiplication by $k$ on the component of grade $k\in K$.
\begin{prop}\plabel{prop:dsw2} ($\mathbb Q\subset K$)
Suppose $W\subset \botimes\mathfrak g$ is  closed for actions of all $\sigma_r\in\Sigma_r$
inducing  permutations in the $r$th tensor order, and it is
also closed for selecting homogeneous components from $\bigotimes\mathfrak g$ position-wise in the tensor products.
Also assume that $\bo m|_{W} :W\rightarrow\mathcal U\mathfrak g$ is injective.
Now, if
\[P=\sum_{n,\lambda}a_{1,\lambda}\otimes\ldots\otimes a_{n,\lambda} \]
such that $P\in W$ and $\bo m(P)\in\mathfrak g$, then
\[\bo m(w(P))=\sum_{\lambda}\,\bo[a_{1,\lambda},\ldots,a_{n-1,\lambda}, w(a_{n,\lambda})\bo]_{\mathrm L}.\]
\begin{proof} We can assume that $a_{i,\lambda}$ is of homogeneous grade $w_{i,\lambda}$.
Then the previous proof works but using  \eqref{eq:EEw} instead of \eqref{eq:EEn}.
\end{proof}
\end{prop}
The application of the centrally bracketed DSW lemma tells
\begin{equation}
n(n-1)\cdot\mu^{\Lie}_n(X_1,\ldots,X_n)
=\sum_{\substack{\sigma\in\Sigma_n,\\ 1\leq p<n}}\mu_\sigma
\bo[\bo[X_{\sigma(1)},\ldots ,X_{\sigma(p)}\bo]_{\mathrm L}, \bo[X_{\sigma(p+1)} ,\ldots,X_{\sigma(n)}
\bo]_{\mathrm R}\bo].
\plabel{eq:FFFn}
\end{equation}
\begin{prop}\plabel{prop:dsw3} ($\mathbb Q\subset K$)
Suppose that the $K$-submodule  $W\subset \botimes^n\mathfrak g$ is  closed for actions of  $\Sigma_n$
inducing  permutations in the order of tensor product.
Also assume that $\bo m|_{W} :W\rightarrow\mathcal U\mathfrak g$ is injective.
Now, if  \[P=\sum_{\lambda}a_{1,\lambda}\otimes\ldots\otimes a_{n,\lambda} \]
such that $P\in W$ and $\bo m(P)\in\mathfrak g$, then
\[n(n-1)\cdot \bo m(P)
=\sum_{\lambda}\sum_{p=1}^{n-1}\,
\bo[\bo[a_{1,\lambda},\ldots,a_{p,\lambda}\bo]_{\mathrm L},
\bo[a_{p+1,\lambda} ,\ldots,a_{n,\lambda} \bo]_{\mathrm R}\bo]\]
\begin{proof}
We can argue as in Proposition \ref{prop:dsw1} but with respect to
\eqref{eq:FFFn}.
\end{proof}
\end{prop}

\snewpage
\section{On Burgunder's splitting}\plabel{sec:Burgunder}
Now we will give a very particular construction to \eqref{eq:splita}--\eqref{eq:basid}.
Assume $P\in\mathrm F[X_\lambda : \lambda\in\Lambda]_{(n)}$; expanded,
\[P=\sum_{(\lambda_1,\ldots,\lambda_n)\in \Lambda^n}c_{(\lambda_1,\ldots,\lambda_n)}X_{\lambda_1}\cdot\ldots\cdot X_{\lambda_n}. \]
By the Dynkin--Specht--Wever lemma,
if $P$ is a commutator polynomial, i. e. $P\in\mathrm F^{\Lie}[X_\lambda : \lambda\in\Lambda]_{(n)}$, then
\begin{equation}
P=\frac1n\sum_{(\lambda_1,\ldots,\lambda_n)\in \Lambda^n}c_{(\lambda_1,\ldots,\lambda_n)}[X_{\lambda_1},[X_{\lambda_2},\ldots, [X_{\lambda_{n-1}},X_{\lambda_n}]\ldots]].
\plabel{eq:dynmulti}
\end{equation}

For this reason, it is natural to define
\begin{equation}
\delta^{\mathrm L}_n(P)=\frac1n\sum_{(\lambda_1,\ldots,\lambda_n)\in \Lambda^n}c_{(\lambda_1,\ldots,\lambda_n)}
X_{\lambda_1}\otimes[X_{\lambda_2},\ldots, [X_{\lambda_{n-1}},X_{\lambda_n}]\ldots].
\plabel{eq:dedef}
\end{equation}
This is well-defined and behaves naturally with respect to substitution of variables, etc.;
it satisfies \eqref{eq:basid}.
[At first sight, the construction here looks like ``heavily tilted to the left'', but this is not so.
If we do the similar construction $\delta^{\mathrm R}_n$ with right-iterated commutators,
then we find that $\delta^{\mathrm L}_n(P)$ and $\delta^{\mathrm R}_n(P)$ are related to each other by the simple switch map
generated by $x\otimes y\mapsto -y\otimes x$.
This follows from either from direct combinatorial arguments, or from the natural symmetries of the
Dynkin--Magnus commutators / Eulerian idempotents (cf. Solomon \cite{SS}, Reutenauer \cite{R}, Dynkin \cite{Dyy}, Magnus \cite{M},
Mielnik, Pleba\'nski \cite{MP});
but it will also be clear from our arguments later.]

More generally, for $n\geq2$, we can define the extended map
\[\tilde\delta^{\mathrm L}_n: \mathrm F[X_\lambda : \lambda\in\Lambda]_{(n)}\rightarrow
\mathrm F^{\Lie}[X_\lambda : \lambda\in\Lambda]_{(1)}\otimes \mathrm F^{\Lie}[X_\lambda : \lambda\in\Lambda]_{(n-1)}\]
by the same formula as in \eqref{eq:dedef} but with extended domain.
Then, of course, \eqref{eq:basid} does not extend.
[Also, after constructing $\tilde\delta^{\mathrm R}_n$ in analogous manner, there is no simple relationship
between $\tilde\delta^{\mathrm L}_n(P)$ and $\tilde\delta^{\mathrm R}_n(P)$.]
Using a more compact notation, we may write, for $P\in \mathrm F[X_\lambda : \lambda\in\Lambda]_{(\geq 2)}$,
\[\tilde\delta^{\mathrm L}_n(P)
=\degg^{-1}\left(\sum_{\lambda\in\Lambda} X_\lambda\otimes \mathfrak D^{\mathrm L}( \partial^{\mathrm L}_{X_\lambda}(P))\right)
=\sum_{\lambda\in\Lambda} X_\lambda\otimes (\degg+\Id)^{-1}\left(\mathfrak D^{\mathrm L}( \partial^{\mathrm L}_{X_\lambda}(P))\right)
.\]
Here $\partial^{\mathrm L}_{X_\lambda}$ selects the terms standing multiplicatively after $X_\lambda$,
$\mathfrak D^{\mathrm L}$ prepares the left-iterated commutators
$X_{\lambda_1}\ldots X_{\lambda_k}\mapsto [X_1,\ldots [X_{\lambda_{k-1}},X_{\lambda_k}]\ldots]$,
$\degg$ multiplies the component of degree $n$ by $n$.
($\partial^{\mathrm L}_{X_\lambda}$ can be thought as the composition $\partial^{\mathrm L}\Starting_{X_\lambda}$,
where $\partial^{\mathrm L}\Starting_{X_\lambda}$ selects associative monomials starting with $X_\lambda$, and
$\partial^{\mathrm L}$ removes the starting variable.)

In Burgunder's version \cite{Bur}, $\hat\delta^{\mathrm L}_n(P)$ is used, where the first
canonical projection / Eulerian idempotent is applies to $P$, then $\delta^{\mathrm L}$ is used.
[That makes  $\hat\delta^{\mathrm L}_n$ closely related to $\hat\delta^{\mathrm R}_n$ again.]
For us, however, $\delta^{\mathrm L}_n$ is perfectly sufficient.
\snewpage
Computing $\delta^{\mathrm L}_n(P)$, in general, may be somewhat cumbersome due to the expansion to be taken.
It would be better if we would have a clear picture of  $\delta^{\mathrm L}_n$ acting on commutator monomials.
For us, a commutator monomial will be just a commutator expression of the variables, possibly multiplied by
$-1$; e. g. $-[[X_1,X_3],[X_2,[X_5,X_4]]]$.
This possible multiplication by $-1$ makes true difference only in homogeneity degree $1$, where $-X_\lambda$ will
also be considered as a commutator monomial (otherwise, the sign change can be realized by a switch in a commutator).
And indeed, a simple description of $\delta^{\mathrm L}_n$ on commutator monomials is possible.
As $\delta^{\mathrm L}_n$ behaves naturally with respect to substitution variables, it is sufficient to consider commutator
monomials where every variable has multiplicity $1$. If $M=\pm[M_1,M_2]$ is such a commutator monomial, and
$X_\lambda$ is one of variables, then we define the co-weight of $X_\lambda$ is $M$ as
\[\cw_{X_\lambda}(M)=\begin{cases}
\deg M_2&\text{ if $X_\lambda$ is a variable of $M_1$,}\\
\deg M_1&\text{ if $X_\lambda$ is a variable of $M_2$.}
\end{cases}\]
Note that this is well-defined, as any such monomial $M$ can be written in a unique way,
apart from switches in the commutators and the corresponding sign changes.
(This is not true if there are multiplicities in the variables).
Under the same multiplicity assumption (uniformly 1)
we can define $\rem^{\mathrm L}_{X_k}(M)$,  the commutator monomial where $X_k$ is
removed from $M$  to the left, as follows:
Taking switches in the commutator monomial $M$, we can write it
as
\begin{equation}
M=(-1)^l[[\ldots[[X_k, N_1],N_2],\ldots, N_{s-1}],N_s],
\plabel{eq:Mek}
\end{equation}
where the $N_i$ are other commutator monomials.
Then
\[\rem^{\mathrm L}_{X_k}(M)=(-1)^l[N_1,[N_2,\ldots[N_{s-1},N_s]\ldots]].\]
(In this setting $\cw_{X_\lambda}(M)=\deg N_s$ .)
[The removal to right $\rem^{\mathrm R}_{X_k}(M)$ can be defined analogously, but
it is not hard to see that it yields only
$\rem^{\mathrm R}_{X_k}(M)=-\rem^{\mathrm L}_{X_k}(M)$.]

Now we can state
\begin{theorem}\plabel{th:monoeff}
Suppose that $M$ is a commutator monomial of the variables $X_1,\ldots,X_n$, each with multiplicity $1$, then
\[\delta^{\mathrm L}_n(M)=\frac1n\sum_{k=1}^n \cw_{X_k}(M)\,X_k\otimes \rem^{\mathrm L}_{X_k}(M).\]
\begin{proof}
It is sufficient to prove
\begin{equation}
\mathfrak D^{\mathrm L}(\partial^{\mathrm L}_{X_k}M)=\cw_{X_k}(M)\rem^{\mathrm L}_{X_k}(M)
\plabel{eq:sek}
\end{equation}
for any variable $X_k$.
 If $M$ is as in \eqref{eq:Mek}, then
\[\partial^{\mathrm L}_{X_k}M=(-1)^lN_1N_2\ldots N_{s-1}N_s.\]
From the properties of the adjoint representation we obtain
\begin{equation}
\mathfrak D^{\mathrm L}(\partial^{\mathrm L}_{X_k}M)=\mathfrak D^{\mathrm L}((-1)^lN_1N_2\ldots N_{s-1}N_s)
=(-1)^l[N_1,[N_2,\ldots[N_{s-1},\mathfrak D^{\mathrm L}(N_s)]\ldots]].
\plabel{eq:rek}
\end{equation}
Now the Dynkin--Specht--Wever lemma says
$\mathfrak D^{\mathrm L}(N_s)=(\deg N_s)\cdot N_s=(\cw_{X_k}(M))\cdot N_s$.
Plugging this \eqref{eq:rek}, we obtain \eqref{eq:sek} immediately.
\end{proof}
\end{theorem}
\begin{example}\plabel{ex:monoeff}
\begin{align*}
\delta^{\mathrm L}_7( [[X_1,[X_2,[X_3,X_4]]],[[X_5,X_6],X_7]])=\frac17\Bigl(
&3\cdot X_1\otimes [[X_2,[X_3,X_4]],[[X_5,X_6],X_7]]   \\
-&3\cdot X_2\otimes [ [X_3,X_4],   [  X_1,[[X_5,X_6],X_7]]] \\
+&3\cdot X_3\otimes [X_4,  [X_2,[ X_1,[[X_5,X_6],X_7]]]]  \\
-&3\cdot X_4\otimes  [ X_3,[X_2,[X_1,[[X_5,X_6],X_7]]]] \\
-&4\cdot X_5\otimes [X_6, [X_7, [X_1,[X_2,[X_3,X_4]]] ]] \\
+&4\cdot X_6\otimes [X_5,[ X_7,[X_1,[X_2,[X_3,X_4]]]]] \\
+&4\cdot X_7\otimes [[X_5,X_6], [X_1,[X_2,[X_3,X_4]]]]
\Bigr).
\qedexer
\end{align*}
\end{example}
As $\delta^{\mathrm L}_n$ is invariant for substitution of variables, this provides
information for monomials even if some variables have multiplicities more than $1$.
Then $\delta^{\mathrm L}_n(P)$ can be recovered by linear extension.

\begin{theorem}\plabel{th:lineff}
 If $n\geq2$, then
\begin{align*}
 \delta^{\mathrm L}_n( [X_1,[X_2,\ldots, &[X_{n-1},X_n]\ldots]])=\\
=\frac1n\biggl(&(n-1) X_1\otimes[X_2,\ldots, [X_{n-1},X_n]\ldots]+ \\
&+\sum_{k=2}^{n-1}   X_k\otimes \bigl[\,[\ldots[X_1,X_2],\ldots,X_{k-1}]\,,\,[X_{k+1},\ldots, [X_{n-1},X_n]\ldots] \,\bigr]  \\
&-X_n\otimes[\ldots[X_1,X_2],\ldots,X_{n-1}]\biggr).
\end{align*}
\begin{proof}
This is immediate from Theorem \ref{th:monoeff}.
\end{proof}
\end{theorem}
\begin{example}\plabel{ex:lineff}
\begin{align*}
\delta^{\mathrm L}_7([X_1,[X_2,[X_3,[X_4,[X_5,[X_6,X_7]]]]]])=\frac17\Bigl(
&6\cdot X_1\otimes [X_2,[X_3,[X_4,[X_5,[X_6,X_7]]]]]\\
&+X_2\otimes [X_1,[X_3,[X_4,[X_5,[X_6,X_7]]]]]\\
&+X_3\otimes [[X_1,X_2],[X_4,[X_5,[X_6,X_7]]]]\\
&+X_4\otimes [[[X_1,X_2],X_3],[X_5,[X_6,X_7]]]\\
&+X_5\otimes [[[[X_1,X_2],X_3],X_4],[X_6,X_7]]\\
&+X_6\otimes [[[[[X_1,X_2],X_3],X_4],X_5],X_7]\\
&-X_7\otimes [[[[[X_1,X_2],X_3],X_4],X_5],X_6]
\Bigr).\qedexer
\end{align*}
\end{example}
In what follows, instead of $\delta^{\mathrm L}_{n}$ we might just write $\delta^{\mathrm L}$.
\begin{cor}\plabel{cor:lineff}
If $P$ is a Lie polynomial, then
\[
\delta^{\mathrm L} ([X_\lambda,[X_\lambda,P]])=X_\lambda\otimes[X_\lambda,P]=\tilde\delta^{\mathrm L} (X_\lambda X_\lambda P).\]
\begin{proof}
$P$ can be assumed to be of left-iterated commutator form, then in Theorem \ref{th:lineff} only the first two commutators survive.
\end{proof}
\end{cor}

\section{Splittings of Dynkin--Burgunder type more generally}
\plabel{sec:DynkinBurgunder}
From the weighted DSW lemma it is clear how to generalize \eqref{eq:dedef}.
Assume that $P\in\mathrm F[X_\lambda : \lambda\in\Lambda] $; expanded,
\[P=\sum_{n\geq2}\sum_{(\lambda_1,\ldots,\lambda_n)\in \Lambda^n}c_{(\lambda_1,\ldots,\lambda_n)}X_{\lambda_1}\cdot\ldots\cdot X_{\lambda_n}. \]

Assume that a weighting $w$ is given such that to any variable
$X_\lambda$ the weight $w_\lambda$ is assigned.
Similarly to \eqref{eq:dedef}, we let
\begin{equation}
\delta^{\mathrm L}_{w }(P)= \sum_{n\geq2}\sum_{(\lambda_1,\ldots,\lambda_n)\in \Lambda^n}\tfrac{w_{\lambda_n}}{w_{\lambda_1}+\ldots+w_{\lambda_n}}\,c_{(\lambda_1,\ldots,\lambda_n)}
X_{\lambda_1}\otimes[X_{\lambda_2},\ldots, [X_{\lambda_{n-1}},X_{\lambda_n}]\ldots],
\plabel{eq:dedef2}
\end{equation}
as long as $w_{\lambda_1}+\ldots+w_{\lambda_n}\neq0$ for $c_{(\lambda_1,\ldots,\lambda_n)}\neq0$.
This definition is invariant for permutation or substitution of variables as long as the weights are the same.
Suppose that $M$ is a commutator monomial of the variables $X_1,\ldots,X_n$, each with multiplicity $1$;
$M=\pm[M_1,M_2]$.
Let $w_M$, $w_{M_1}$, and  $w_{M_2}$ be sums of the weights in the respective monomials ($w_M=w_{M_1}+w_{M_2}$).
We define the co-weight of $X_\lambda$ in $M$ as
\[\cw^w_{X_\lambda}(M)=\begin{cases}
w_{M_2}&\text{ if $X_\lambda$ is a variable of $M_1$,}\\
w_{M_1}&\text{ if $X_\lambda$ is a variable of $M_2$.}
\end{cases}\]
The proof and the corollary of the following theorem are like in case of Theorem \ref{th:monoeff}:
\begin{theorem}\plabel{th:monoeff2}
Suppose that $M$ is a commutator monomial of the variables $X_1,\ldots,X_n$, each with multiplicity $1$, then
\[\delta^{\mathrm L}_w(M)=
\frac{1}{w_M}\sum_{k=1}^n  \,\cw^w_{X_k}(M)\,X_k\otimes \rem^{\mathrm L}_{X_k}(M),\]
as long as $w_M\neq0$.
\qed
\end{theorem}
The notation $\delta_{X}^{\mathrm L}$ will be used for $w=\degg_X$.
\begin{cor}\plabel{cor:lineff2}
If $P$ is a Lie polynomial and $\delta^{\mathrm L}_w$ applies, then
\[
\delta^{\mathrm L}_w ([X_\lambda,[X_\lambda,P]])=X_\lambda\otimes[X_\lambda,P]=\tilde\delta_w^{\mathrm L} (X_\lambda X_\lambda P).
\eqed\]
\end{cor}
\snewpage
\section{Relation to the Kashiwara--Vergne problem}\plabel{sec:reKV}

Before proceeding, let us make the following
\begin{remark}\plabel{rem:trace}
One can notice that in the formulation of the Kashiwara--Vergne problem in the Introduction,
 in point (b), we have apparently deviated from the customary formulation, cf. Kashiwara--Vergne \cite{KV}, etc.
Nevertheless, it is an equivalent formulation in the universal setting.
The formal trace abstracts the original trace by the formula
``$\tr\left(Z\mapsto(\ad Y_1)\ldots(\ad Y_n)Z \right)$''{}$=\Tr(Y_1\ldots Y_n)$.
Taking the commutator expansion of $(\ad Y_1)\ldots(\ad Y_n)Z$, we see that the result is determined by the $Z$-ending
associative monomials.
If $P(Y_1,\ldots,Y_n,Z)$ is a commutator polynomial where the multiplicity of $Z$ is $1$, then
``$\tr\left(Z\mapsto P(Y_1,\ldots,Y_n,Z) \right)$''{}$=\Tr \partial^{\mathrm R}_ZP(Y_1,\ldots,Y_n,Z)$.
Considering the commutator polynomial  $Q(Y_1,\ldots,Y_n,X)$  where  the multiplicity of $X$ is $1$,
we see that ``$\tr\left(Z\mapsto  (\ad X)\left.\frac{\mathrm d}{\mathrm dt}Q(Y_1,\ldots,Y_n,X+tZ)\right|_{t=0} \right)$''$=$
``$\tr\left(Z\mapsto  (\ad X)Q(Y_1,\ldots,Y_n, Z)  \right)$''$=$\break { }
$\Tr X\partial^{\mathrm R}_ZQ(Y_1,\ldots,Y_n,Z)${}$
=\Tr \partial^{\mathrm R}_ZQ(Y_1,\ldots,Y_n,Z)X=\Tr \mathrm E^{\mathrm R}_XQ(Y_1,\ldots,Y_n,X)$. Actually,
\[\text{``$\tr\left(Z\mapsto  (\ad X)\left.\frac{\mathrm d}{\mathrm dt}Q(Y_1,\ldots,Y_n,X+tZ)\right|_{t=0} \right)$''}
=\Tr \mathrm E^{\mathrm R}_XQ(Y_1,\ldots,Y_n,X)\]
remains valid even if the multiplicity of $X$ is not necessarily $1$:
All one has to do is to non-deterministically polarize $X$ in $Q$ to variables $X_\lambda$, take the formula above for all
$X_\lambda$, and add them up, then forget the indices in the $X_\lambda$.
\qedremark
\end{remark}
Let us rephrase the KV problem:
It postulates the existence of analytic Lie series $A(X,Y)$ and $B(X,Y)$
such that

(a) $X+Y-\BCH(Y,X)=[X,A(X,Y)]+[Y,B(X,Y)]$,

(b) $\Tr\left( \Ending_X \left( \beta(-\ad X)A(X,Y)\right)+\Ending_Y \left( \beta(\ad Y) B(X,Y)\right) \right)=$\\
\phantom{ccccccccccccccccccccccccccccccccccccccccc} $=
\frac12\Tr\left( \beta(X)+\beta(Y)-\beta(\BCH(X,Y))-1 \right)$,\\
 where
  $\beta(u)=\frac{u}{\mathrm e^{u}-1}$ is a formal power series;
 $\ad S: W\mapsto[S,W]$ is the well-know adjoint action; $\Tr$ is the formal trace
 sending associative monomials to cyclically symmetrized monomials;
 $\Ending_X$ and $\Ending_Y$ select the $X$-ending and $Y$-ending associative monomials, respectively.

One can rewrite part (b) using the following observations:
In general, $\Ending_X  C(X,Y)+\Ending_Y  C(X,Y)=C(X,Y)$,
Thus, if $C(X,Y)$ is a Lie series, then
\begin{equation}
\Tr\left(\Ending_X  C(X,Y)+\Ending_Y  C(X,Y)\right)=\Tr C(X,Y)=0.
\plabel{eq:lietrace}
\end{equation}
Note that $\beta(u)=\eta(u)-\frac12 u$, where $\eta(u)=1+\frac1{12}u^2+\ldots$ is a function $u^2$.
Then
\[\Tr\left( \Ending_X \left( \beta(-\ad X)A(X,Y)\right) \right)=
\Tr\left( \Ending_X \left( \eta(\ad X)A(X,Y)\right) \right)+
\Tr\left( \frac12\Ending_X  [X,A(X,Y)] \right).
\]
Similarly,
\[\Tr\left( \Ending_Y \left( \beta( \ad Y)B(X,Y)\right) \right)=
\Tr\left( \Ending_Y \left( \eta( \ad Y)B(X,Y)\right) \right)-
\Tr\left( \frac12\Ending_Y  [Y,B(X,Y)] \right).\]
Now,
\[\Tr\left( \Ending_Y  [Y,B(X,Y)] \right)=-\Tr\left( \Ending_X  [Y,B(X,Y)] \right).\]
Therefore,
\[\Tr\left( \frac12\Ending_X  [X,A(X,Y)] \right)-
\Tr\left( \frac12\Ending_Y  [Y,B(X,Y)] \right)=
\frac12\Tr\left(  \Ending_X  [X,A(X,Y)]+ \Ending_X  [Y,B(X,Y)]\right)\]

\[=\frac12\Tr\Ending_X \left(  X+Y-\BCH(Y,X)\right)\]

Consequently, condition KV(b) can be written as

(b)'
$\Tr\left( \Ending_X \left( \eta(\ad X)A(X,Y)\right) \right)+\Tr\left( \Ending_Y \left( \eta( \ad Y)B(X,Y)\right) \right)=$\\
\phantom{fwwwww}
$=\frac12\Tr\left( \beta(X)+\beta(Y)-\beta(\BCH(X,Y))-1 \right)-\frac12\Tr\Ending_X \left(  X+Y-\BCH(Y,X)\right)$.

As $\eta$ is an even function, one can see form (a),(b)' that the KV condition is decoupled in even and odd orders.
(Cf.  Rouvière \cite{Rou}.)
I. e. the Kashiwara--Vergne problem is sufficient to solve separately for the pairs $\left(A(X,Y)_{\mathrm{even}},B(X,Y)_{\mathrm{even}}\right)$
and $\left(A(X,Y)_{\mathrm{odd}},B(X,Y)_{\mathrm{odd}}\right)$.

\begin{theorem}\plabel{th:oddkv}
Let u set $A_X(X,Y)_{\mathrm{odd}}$ and $B_X(X,Y)_{\mathrm{odd}}$ such that
\[X\otimes A_X(X,Y)_{\mathrm{odd}}+Y\otimes B_X(X,Y)_{\mathrm{odd}}
=\delta^{\mathrm L}_X\left(\left(X+Y-\BCH(Y,X)\right)_{\mathrm{even}}\right);\]
and let us set $A_Y(X,Y)_{\mathrm{odd}}$ and $B_Y(X,Y)_{\mathrm{odd}}$ such that
\[X\otimes A_Y(X,Y)_{\mathrm{odd}}+Y\otimes B_Y(X,Y)_{\mathrm{odd}}
=\delta^{\mathrm L}_Y\left(\left(X+Y-\BCH(Y,X)\right)_{\mathrm{even}}\right).\]
Then we, claim, for either choice $V=X,Y$,

(a) $(X+Y-\BCH(Y,X))_{\mathrm{even}} =[X,A_V(X,Y)_{\mathrm{odd}}]+[Y,B_V(X,Y)_{\mathrm{odd}}]$,

(b')
$\Tr\left( \Ending_X \left( \eta(\ad X)A_V(X,Y)_{\mathrm{odd}}\right) \right)
+\Tr\left( \Ending_Y \left( \eta( \ad Y)B_V(X,Y)_{\mathrm{odd}}\right) \right)=$\\
\phantom{fwwwww}
$=\frac12\Tr\left( \beta(X)+\beta(Y)-\beta(\BCH(X,Y))-1 \right)_{\mathrm{odd}}
-\frac12\Tr\Ending_X \left(  X+Y-\BCH(Y,X)_{\mathrm{odd}}\right)$
\\holds.
\end{theorem}
\begin{proof}
In order to prove Theorem \ref{th:oddkv} it is sufficient to prove the identities in
\begin{lemma}\plabel{lem:oddkv}
\begin{equation}\eta(\ad X)A_X(X,Y)_{\mathrm{odd}}=0 ;\plabel{eq:AX}\end{equation}
\begin{equation}\eta(\ad Y)B_X(X,Y)_{\mathrm{odd}}=\frac12(Y-\BCH(Y,X))_{\mathrm{odd}} ;\plabel{eq:BX}\end{equation}
\begin{equation}\eta(\ad X)A_Y(X,Y)_{\mathrm{odd}}=-\frac12(X-\BCH(Y,X))_{\mathrm{odd}} ;\plabel{eq:AY}\end{equation}
\begin{equation}\eta(\ad Y)B_Y(X,Y)_{\mathrm{odd}}=0 ;\plabel{eq:BY}\end{equation}
\begin{equation}\Tr\left( \beta(X)+\beta(Y)-\beta(\BCH(X,Y))-1 \right)_{\mathrm{odd}}=0 .\plabel{eq:etrace}\end{equation}
\end{lemma}
Then (a) follows from  the construction of the pairs $\left(A_V(X,Y)_{\mathrm{odd}},B_V(X,Y)_{\mathrm{odd}}\right)$,
 and (b') follows from the vanishing of most terms and from \eqref{eq:lietrace} (available for Lie series).
Therefore we are left to prove Lemma \ref{lem:oddkv}.
This will be done in the next section, using the resolvent formalism.
\qedno
\end{proof}
 \snewpage
\section{On the (discrete) resolvent method}\plabel{sec:discRes}
If $A$ is sufficiently close to $A$, then for $\lambda\in[0,1]$, we may define the resolvent expression
\[\mathcal R^{(\lambda)}(A)=\frac{A-1}{\lambda+(1-\lambda)A}.\]
Note that if $A$ is a formal perturbation of $1$, then this works out, and
 $\mathcal R^{(\lambda)}(A)$  is a formal perturbation of $0$, cf. \eqref{eq:expert}.
It is useful to know that, if $A$ invertible, then
\begin{equation}
\mathcal R^{(\lambda)}(A^{-1})=-\mathcal R^{(1-\lambda)}(A)
\plabel{eq:invres}
\end{equation}
hold if either side makes sense.
But this is again so if $A$ is a formal perturbation of $1$.
In whatever way we define $\log$ around $1$, we find
\[\mathcal R^{(\lambda)}(A)=\dfrac{\mathrm d}{\mathrm d\lambda}\left(\log\dfrac{A}{\lambda+(1-\lambda)A}\right).\]
Thus, for $A$ near $1$,
\begin{equation}
\log A=\int_{\lambda=0}^1 \mathcal R^{(\lambda)}(A)\,\mathrm d\lambda.
\plabel{eq:log}
\end{equation}

Here we will be interested merely in the case when $A$ is a formal perturbation of $1$.
For $A=\exp X$, we find
\begin{equation}
\mathcal R^{(\lambda)}(\exp X)=X+\left(\lambda-\frac12\right)X^2+\ldots\plabel{eq:expert}
\end{equation}
Here \eqref{eq:log} yields
\[X=\int_{\lambda=0}^1 \mathcal R^{(\lambda)}(\exp X)\,\mathrm d\lambda.\]
For the Baker--Campbell--Hausdorff expression, \eqref{eq:log} yields
\begin{equation}
\BCH(X,Y)=\int_{\lambda=0}^1 \mathcal R^{(\lambda)}((\exp X)(\exp Y) )\,\mathrm d\lambda.
\plabel{eq:BCH}
\end{equation}
Using \eqref{eq:invres}, or simply just the properties of $\log$, we see that $\BCH(-X,-Y)=-\BCH(Y,X)$.
In particular, $\frac12\BCH(X,Y)+\frac12\BCH(Y,X)=\BCH(X,Y)_{\mathrm{odd}}$.

Let us recall the formal series
\[\beta(x)=\frac{x}{\mathrm e^x-1}=1-\frac{1}{2}x+\frac{1}{12}{x}^{2}-{\frac {1}{720}}{x}^{4}+\ldots\,.\]
Using elementary calculus, it is easy to see that
\begin{equation}
\beta(X)=1+\int_{\lambda=0}^1  (\lambda-1)\,\mathcal R^{(\lambda)}(\exp X)\,\mathrm d\lambda.
\plabel{eq:betares}
\end{equation}
\begin{commentx}
We will also consider the function
\[\gamma(x)=\beta'(x)+\frac12;\]
Analytically, it has similar properties as $\beta$, but $\gamma$ is an odd function.

\begin{lemma}
\[X=\int_{\lambda=0}^1\mathcal R^{(\lambda)}(\exp X)\,\mathrm d\lambda\]
\[\beta(X)+X-1=\beta(-X)-1=\int_{\lambda=0}^1  \lambda\,\mathcal R^{(\lambda)}(\exp X)\,\mathrm d\lambda\]
\[\beta(X)-1=\int_{\lambda=0}^1  (\lambda-1)\,\mathcal R^{(\lambda)}(\exp X)\,\mathrm d\lambda\]
\[-\gamma(X)=\int_{\lambda=0}^1  \lambda(\lambda-1)\,\mathcal R^{(\lambda)}(\exp X)\,\mathrm d\lambda\]
\[-1+X+\beta(X)-\gamma(X)=\int_{\lambda=0}^1\lambda^2\,\mathcal R^{(\lambda)}(\exp X)\,\mathrm d\lambda\]
\[1-\beta(X)-\gamma(X)=\int_{\lambda=0}^1(\lambda-1)^2\,\mathcal R^{(\lambda)}(\exp X)\,\mathrm d\lambda\]
\end{lemma}
\end{commentx}
\begin{lemma} It holds that
\begin{equation*}
\beta(\BCH(X,Y))_{\mathrm{odd}}=\frac12\left(-\BCH(X,Y)_{\mathrm{odd}}+\mathrm W(X,Y)\right),
\end{equation*}
where
\[\mathrm W(X,Y)=
\int_{\lambda=0}^1   \left(\lambda-\tfrac12\right) \,\mathcal R^{( \lambda)}((\exp X)(\exp Y))\,\mathrm d\lambda
-  \int_{\lambda=0}^1   \left(\lambda-\tfrac12\right) \,\mathcal R^{( \lambda)}((\exp Y)(\exp X))\,\mathrm d\lambda.\]

$\mathrm W(X,Y)$ is traceless.
Furthermore, and consequently,
\[ \left( \beta(X)+\beta(Y)-\beta(\BCH(X,Y))-1 \right)_{\mathrm{odd}}= -\tfrac12(X+Y-\BCH(X,Y))_{\mathrm{odd}}-\mathrm W(X,Y)\]
is traceless, proving \eqref{eq:etrace}.

\begin{proof} By \eqref{eq:betares}, and by \eqref{eq:invres}, then changing $\lambda$ to $1-\lambda$, we obtain
\begin{align*}
\beta(\BCH(-X,-Y))&=1+\int_{\lambda=0}^1  (\lambda-1)\,\mathcal R^{(\lambda)}((\exp -X)(\exp -Y))\,\mathrm d\lambda\\
&=1+\int_{\lambda=0}^1  (1-\lambda)\,\mathcal R^{(1-\lambda)}((\exp Y)(\exp X))\,\mathrm d\lambda\\
&=1+\int_{\lambda=0}^1   \lambda \,\mathcal R^{( \lambda)}((\exp Y)(\exp X))\,\mathrm d\lambda.
\end{align*}
Therefore,
\begin{align*}
\beta(&\BCH(X,Y))_{\mathrm{odd}}\\&=\frac12\left(\beta(\BCH(X,Y))-\beta(\BCH(-X,-Y)) \right)\\
&=\frac12\left( \int_{\lambda=0}^1  (\lambda-1)\,\mathcal R^{(\lambda)}((\exp X)(\exp Y))\,\mathrm d\lambda
-  \int_{\lambda=0}^1   \lambda \,\mathcal R^{( \lambda)}((\exp Y)(\exp X))\,\mathrm d\lambda
  \right)\\
  &=\frac12\Biggl( -\frac12\BCH(X,Y)-\frac12\BCH(Y,X)+\\
  &\quad+\int_{\lambda=0}^1   \left(\lambda-\tfrac12\right) \,\mathcal R^{( \lambda)}((\exp X)(\exp Y))\,\mathrm d\lambda
-  \int_{\lambda=0}^1   \left(\lambda-\tfrac12\right) \,\mathcal R^{( \lambda)}((\exp Y)(\exp X))\,\mathrm d\lambda
  \Biggr).
\end{align*}
This simplifies as indicated.
$\mathrm W(X,Y)$ is traceless because on term is conjugated to the other by an exponential.
$\left(X+Y-\BCH(X,Y)\right)_{\mathrm{odd}}$ is traceless because it is a commutator expression;
this proves the statement.
\end{proof}
\end{lemma}

\begin{lemma}\plabel{lem:start1}
\[\Starting_X\mathcal R^{(\lambda)}((\exp X)(\exp Y) )=(\exp(X)-1)\exp(Y)(\lambda+(1-\lambda)(\exp X)(\exp Y))^{-1};\]
\[\Starting_Y\mathcal R^{(\lambda)}((\exp X)(\exp Y) )=(\exp(Y)-1)(\lambda+(1-\lambda)(\exp X)(\exp Y))^{-1}.\]
\begin{proof}
The LHS in the first equation starts with $X$, the LHS the second equation starts with $Y$,
yet their sum is obviously $\mathcal R^{(\lambda)}((\exp X)(\exp Y) )$.
This proves the statement.
\end{proof}
\end{lemma}

\begin{lemma}\plabel{lem:start2}
\[\Starting_X\mathcal R^{(\lambda)}((\exp X)(\exp Y) )=
\mathcal R^{(\lambda)}(\exp X ) + \lambda\mathcal R^{(\lambda)}(\exp X )\Starting_Y\mathcal R^{(\lambda)}((\exp X)(\exp Y) );\]
\[\Starting_Y\mathcal R^{(\lambda)}((\exp X)(\exp Y) )=
\mathcal R^{(\lambda)}(\exp Y ) + (\lambda-1)\mathcal R^{(\lambda)}(\exp Y )\Starting_X\mathcal R^{(\lambda)}((\exp X)(\exp Y) ) .\]
\begin{proof}
Regarding the first equation: After multiplying by
$(\lambda+(1-\lambda)(\exp X))$ on the left,
and by $(\lambda+(1-\lambda)(\exp X)(\exp Y))$ on the right,
the equation checks out according to Lemma \ref{lem:start1}.
The second equation can be proven similarly.
\end{proof}
\end{lemma}
\begin{lemma}\plabel{lem:powres}
(Resolvent decomposition, formal.)
If $X,Y$ are formal variables, then
\begin{align}
\mathcal R^{(\lambda)}((\exp X)(\exp Y))=\sum_{k=0}^\infty\biggl(
&\lambda^k(\lambda-1)^k\mathcal R^{(\lambda)}(\exp Y)(\mathcal R^{(\lambda)}(\exp X)\mathcal R^{(\lambda)}(\exp Y))^{k}
\plabel{eq:dresdec}\\
+&\lambda^k(\lambda-1)^k\mathcal R^{(\lambda)}(\exp X)(\mathcal R^{(\lambda)}(\exp Y)\mathcal R^{(\lambda)}(\exp X))^{k}\notag\\
+&\lambda^{k+1}(\lambda-1)^{k}(\mathcal R^{(\lambda)}(\exp X)\mathcal R^{(\lambda)}(\exp Y))^{k+1} \notag\\
+&\lambda^{k}(\lambda-1)^{k+1}(\mathcal R^{(\lambda)}(\exp Y)\mathcal R^{(\lambda)}(\exp X))^{k+1} \qquad\,\,\biggr).\notag
 \end{align}

In other terms, on the left, we have the sum of non-empty interlaced products of $\mathcal R^{(\lambda)}(\exp X)$ and $\mathcal R^{(\lambda)}(\exp Y)$
with an extra multiplier $\lambda$ for any consecutive term `$\mathcal R^{(\lambda)}(\exp X)\mathcal R^{(\lambda)}(\exp Y)$'
and with an extra multiplier $\lambda-1$ for any consecutive term `$\mathcal R^{(\lambda)}(\exp Y)\mathcal R^{(\lambda)}(\exp X)$'.
\begin{proof}[Note]
One can replace $\exp X$ and $\exp Y$ by other formal perturbations of $1$.
(Indeed, this follows from a simple of change variables.)
\end{proof}
\begin{proof}
This follows from applying Lemma \ref{lem:start2} recursively.
\end{proof}
\snewpage
\begin{commentx}
\begin{proof}[Alternative proof]
 Let ${  x}=1-(\exp X)$, ${  y}=1-(\exp Y)^{-1}$. Then
\begin{align}\mathcal R((\exp X)&(\exp Y),\lambda)=\notag\\
&=\frac{(\exp X)(\exp Y)-1}{\lambda+(1-\lambda)(\exp X)(\exp Y)}\notag\\
&=((\exp X)(\exp Y)-1)(\exp Y )^{-1}(\exp Y)(\lambda+(1-\lambda)(\exp X)(\exp Y))^{-1}\notag\\
&=((\exp X)-(\exp Y)^{-1})(\lambda(\exp Y)^{-1}+(1-\lambda)(\exp X))^{-1}\notag\\
&=({  y}-{  x})(1-(1-\lambda){  x}-\lambda {  y})^{-1}.\notag
\end{align}
Using the formal expansion
\begin{align*}\frac1{1-\tilde {  x}-\tilde {  y}}&=1+\sum_{k=1}^\infty\biggl(
\left(\frac{\tilde {  x}}{1-\tilde {  x}}\frac{\tilde {  y}}{1-\tilde {  y}}\right)^k+
\left(\frac{\tilde {  y}}{1-\tilde {  y}}\frac{\tilde {  x}}{1-\tilde {  x}}\right)^k+\\
&\quad\underbrace{\quad\qquad\qquad+
\frac{\tilde {  y}}{1-\tilde {  y}}
\left(\frac{\tilde {  x}}{1-\tilde {  x}}\frac{\tilde {  y}}{1-\tilde {  y}}\right)^{k-1}+
\frac{\tilde {  x}}{1-\tilde {  x}}
\left(\frac{\tilde {  y}}{1-\tilde {  y}}\frac{\tilde {  x}}{1-\tilde {  x}}\right)^{k-1}
\biggr)}_\text{interlaced products of $\frac{\tilde {  x}}{1-\tilde {  x}}$ and $\frac{\tilde {  y}}{1-\tilde {  y}}$},
\notag
\end{align*}
with $\tilde x=(1-\lambda)x$ and $\tilde y=\lambda y$,
we find
\begin{align}\mathcal R((\exp X)&(\exp Y),\lambda)=\notag\\
&={  y}\left(1+ \frac{\lambda {  y}}{1-\lambda {  y}}\right)\!\!\!\!\underbrace{\left(1+
\frac{(1-\lambda) {  x}}{1-(1-\lambda) {  x}}
+\frac{(1-\lambda) {  x}}{1-(1-\lambda) {  x}}\frac{\lambda {  y}}{1-\lambda {  y}}+\ldots
\right)}_\text{interlaced products of $ \frac{(1-\lambda) {  x}}{1-(1-\lambda) {  x}}$ and $\frac{\lambda {  y}}{1-\lambda {  y}}$
not starting with $\frac{\lambda {  y}}{1-\lambda {  y}}$ }\notag\\
&\quad-{  x}\left(1+ \frac{(1-\lambda) {  x}}{1-(1-\lambda) {  x}}\right)\!\!\!\!\!\!\!\!\!\!\!\!\!\!\!\!\underbrace{\left(1+
\frac{\lambda {  y}}{1-\lambda {  y}}
+\frac{\lambda {  y}}{1-\lambda {  y}}\frac{(1-\lambda) {  x}}{1-(1-\lambda) {  x}}+\ldots
\right)}_\text{interlaced products of $ \frac{(1-\lambda) {  x}}{1-(1-\lambda) {  x}}$ and $\frac{\lambda {  y}}{1-\lambda {  y}}$
not starting with $\frac{(1-\lambda) {  x}}{1-(1-\lambda) {  x}} $ }\notag\\
&=\frac{{  y}}{1-\lambda {  y}}\left(1+
\frac{(1-\lambda) {  x}}{1-(1-\lambda) {  x}}
+\frac{(1-\lambda) {  x}}{1-(1-\lambda) {  x}}\frac{\lambda {  y}}{1-\lambda {  y}}+\ldots
\right)\notag\\
&\quad-\frac{ {  x}}{1-(1-\lambda) {  x}}\left(1+
\frac{\lambda {  y}}{1-\lambda {  y}}
+\frac{\lambda {  y}}{1-\lambda {  y}}\frac{(1-\lambda) {  x}}{1-(1-\lambda) {  x}}+\ldots
\right).\notag
\end{align}
Notice that
\[\frac {  y}{1-\lambda {  y}}=\mathcal R(\exp Y,\lambda),\qquad
\frac {  x}{1-(1-\lambda) {  x}}=-\mathcal R(\exp X,\lambda);\]
yielding the equality in the statement.
\end{proof}
\end{commentx}

\end{lemma}
If   we plug \eqref{eq:dresdec} into \eqref{eq:BCH},
and expand $\mathcal R(\exp X,\lambda)$ and $\mathcal R(\exp Y,\lambda)$ as power series in $X$ and $Y$, respectively, then
we obtain the formula of Goldberg \cite{G}.
(This is explained in greater detail in \cite{LL1}.)
\snewpage
\begin{proof}[Proof of \eqref{eq:AX} and \eqref{eq:BX}.]
(a) We start with the proof of \eqref{eq:AX}.
\begin{align*}
(\degg_X&+\Id)A_X(X,Y)_{\mathrm{odd}}=\\
=&\,\left(\partial_X^{\mathrm L}\mathrm E_X^{\mathrm R}(X+Y-\BCH(Y,X))\right)_{\mathrm{odd}}=\\
=&-\mathfrak D^{\mathrm L}\int_{\lambda=0}^1\left(\frac1X\sum_{k=1}^\infty\lambda^k(\lambda-1)^k
\mathcal R^{(\lambda)}(\exp X)(\mathcal R^{(\lambda)}(\exp Y)\mathcal R^{(\lambda)}(\exp X))^{k}\right)_{\mathrm{odd}}\mathrm d\lambda\\
=&-\frac12\mathfrak D^{\mathrm L}\int_{\lambda=0}^1\left(\frac1X\sum_{k=1}^\infty\lambda^k(\lambda-1)^k\mathcal R^{(\lambda)}(\exp X)(\mathcal R^{(\lambda)}(\exp Y)\mathcal R^{(\lambda)}(\exp X))^{k}\right) \mathrm d\lambda\\
&+\frac12\mathfrak D^{\mathrm L}\int_{\lambda=0}^1\left(\frac1{-X}\sum_{k=1}^\infty\lambda^k(\lambda-1)^k\mathcal R^{(\lambda)}(\exp -X)(\mathcal R^{(\lambda)}(\exp -Y)\mathcal R^{(\lambda)}(\exp -X))^{k}\right) \mathrm d\lambda\\
=&-\frac12\mathfrak D^{\mathrm L}\int_{\lambda=0}^1\left(\frac1X\sum_{k=1}^\infty\lambda^k(\lambda-1)^k\mathcal R^{(\lambda)}(\exp X)(\mathcal R^{(\lambda)}(\exp Y)\mathcal R^{(\lambda)}(\exp X))^{k}\right) \mathrm d\lambda\\
&+\frac12\mathfrak D^{\mathrm L}\int_{\lambda=0}^1\left(\frac1X\sum_{k=1}^\infty\lambda^k(\lambda-1)^k\mathcal R^{(1-\lambda)}(\exp X)(\mathcal R^{(1-\lambda)}(\exp Y)\mathcal R^{(1-\lambda)}(\exp X))^{k}\right) \mathrm d\lambda\\
=&-\frac12\mathfrak D^{\mathrm L}\int_{\lambda=0}^1\left(\frac1X\sum_{k=1}^\infty\lambda^k(\lambda-1)^k\mathcal R^{(\lambda)}(\exp X)(\mathcal R^{(\lambda)}(\exp Y)\mathcal R^{(\lambda)}(\exp X))^{k}\right) \mathrm d\lambda\\
&+\frac12\mathfrak D^{\mathrm L}\int_{\lambda=0}^1\left(\frac1X\sum_{k=1}^\infty(1-\lambda)^k(-\lambda)^k\mathcal R^{( \lambda)}(\exp X)(\mathcal R^{( \lambda)}(\exp Y)\mathcal R^{( \lambda)}(\exp X))^{k}\right) \mathrm d\lambda\\
=&\,0.
\end{align*}
As $\degg_X+\Id$ is invertible, this proves the $A_X(X,Y)_{\mathrm{odd}}=0$, and thus the statement.

(b) Next, we prove \eqref{eq:BX}.
\begin{align*}
\degg_X& B_X(X,Y)_{\mathrm{odd}}=\\
=&\left(\partial_Y^{\mathrm L}\mathrm E_X^{\mathrm R}(X+Y-\BCH(Y,X))\right)_{\mathrm{odd}}\\
=&-\mathfrak D^{\mathrm L}\int_{\lambda=0}^1\left(\frac1Y\sum_{k=0}^\infty
\lambda^{k+1}(\lambda-1)^{k}(\mathcal R^{(\lambda)}(\exp Y)\mathcal R^{(\lambda)}(\exp X))^{k+1}
\right)_{\mathrm{odd}}\mathrm d\lambda\\
=&-\frac12\mathfrak D^{\mathrm L}\int_{\lambda=0}^1\left(\frac1Y\sum_{k=0}^\infty
\lambda^{k+1}(\lambda-1)^{k}(\mathcal R^{(\lambda)}(\exp Y)\mathcal R^{(\lambda)}(\exp X))^{k+1}
\right)\mathrm d\lambda\\
&+\frac12\mathfrak D^{\mathrm L}\int_{\lambda=0}^1\left(\frac1{-Y}\sum_{k=0}^\infty
\lambda^{k+1}(\lambda-1)^{k}(\mathcal R^{(\lambda)}(\exp -Y)\mathcal R^{(\lambda)}(\exp -X))^{k+1}
\right)\mathrm d\lambda\\
=&-\frac12\mathfrak D^{\mathrm L}\int_{\lambda=0}^1\left(\frac1Y\sum_{k=0}^\infty
\lambda^{k+1}(\lambda-1)^{k}(\mathcal R^{(\lambda)}(\exp Y)\mathcal R^{(\lambda)}(\exp X))^{k+1}
\right)\mathrm d\lambda\\
&+\frac12\mathfrak D^{\mathrm L}\int_{\lambda=0}^1\left(\frac1{-Y}\sum_{k=0}^\infty
\lambda^{k+1}(\lambda-1)^{k}(\mathcal R^{(1-\lambda)}(\exp Y)\mathcal R^{(1-\lambda)}(\exp X))^{k+1}
\right)\mathrm d\lambda\\
=&-\frac12\mathfrak D^{\mathrm L}\int_{\lambda=0}^1\left(\frac1Y\sum_{k=0}^\infty
\lambda^{k+1}(\lambda-1)^{k}(\mathcal R^{(\lambda)}(\exp Y)\mathcal R^{(\lambda)}(\exp X))^{k+1}
\right)\mathrm d\lambda\\
&+\frac12\mathfrak D^{\mathrm L}\int_{\lambda=0}^1\left(\frac1{-Y}\sum_{k=0}^\infty
(1-\lambda)^{k+1}(-\lambda)^{k}(\mathcal R^{( \lambda)}(\exp Y)\mathcal R^{( \lambda)}(\exp X))^{k+1}
\right)\mathrm d\lambda\\
 =&-\frac12\mathfrak D^{\mathrm L}\int_{\lambda=0}^1\left(\frac1Y\sum_{k=0}^\infty
\lambda^{k+1}(\lambda-1)^{k}(\mathcal R^{(\lambda)}(\exp Y)\mathcal R^{(\lambda)}(\exp X))^{k+1}
\right)\mathrm d\lambda\\
&+\frac12\mathfrak D^{\mathrm L}\int_{\lambda=0}^1\left(\frac1{Y}\sum_{k=0}^\infty
(\lambda-1)^{k+1}\lambda^{k}(\mathcal R^{( \lambda)}(\exp Y)\mathcal R^{( \lambda)}(\exp X))^{k+1}
\right)\mathrm d\lambda\\
=&-\frac12\mathfrak D^{\mathrm L}\int_{\lambda=0}^1\left(\frac1Y\sum_{k=0}^\infty
\lambda^{k}(\lambda-1)^{k }(\mathcal R^{(\lambda)}(\exp Y)\mathcal R^{(\lambda)}(\exp X))^{k+1}
\right)\mathrm d\lambda.
\end{align*}
As $\degg_X$ commutes with $\ad Y$, we find
\begin{align*}
\degg_X& \eta(\ad Y)B_X(X,Y)_{\mathrm{odd}}=\\
=&-\frac12\mathfrak D^{\mathrm L}\int_{\lambda=0}^1\left(\frac{\eta(Y)}Y\sum_{k=0}^\infty
\lambda^{k}(\lambda-1)^{k } (\mathcal R^{(\lambda)}(\exp Y)\mathcal R^{(\lambda)}(\exp X))^{k+1}
\right)\mathrm d\lambda\\
=&-\frac12\mathfrak D^{\mathrm L}\int_{\lambda=0}^1\left(\frac12\cdot\frac{\exp(Y)+1}{\exp(Y)-1}\sum_{k=0}^\infty
\lambda^{k}(\lambda-1)^{k } (\mathcal R^{(\lambda)}(\exp Y)\mathcal R^{(\lambda)}(\exp X))^{k+1}
\right)\mathrm d\lambda.
\end{align*}
This, however, must be compared to another expression.
From
\[\BCH(X,Y)_{\mathrm{odd}}=\frac12\BCH(X,Y)-\frac12\BCH(-X,-Y)=\frac12\BCH(X,Y)+\frac12\BCH(Y,X),\]
we obtain
\begin{align*}
(Y-\BCH(X,Y))_{\mathrm{odd}}
=-\int_{\lambda=0}^1&\sum_{k=1}^\infty\biggl(
\lambda^k(\lambda-1)^k\mathcal R^{(\lambda)}(\exp Y)(\mathcal R^{(\lambda)}(\exp X)\mathcal R^{(\lambda)}(\exp Y))^{k}\biggr)\\
+&\sum_{k=0}^\infty
\biggl(\lambda^k(\lambda-1)^k\mathcal R^{(\lambda)}(\exp X)(\mathcal R^{(\lambda)}(\exp Y)\mathcal R^{(\lambda)}(\exp X))^{k}\notag\\
&+ \lambda^{k}(\lambda-1)^{k}(\lambda-\tfrac12)(\mathcal R^{(\lambda)}(\exp X)\mathcal R^{(\lambda)}(\exp Y))^{k+1} \notag\\
&+\lambda^{k}(\lambda-1)^{k}(\lambda-\tfrac12)(\mathcal R^{(\lambda)}(\exp Y)\mathcal R^{(\lambda)}(\exp X))^{k+1} \biggr)
\mathrm d\lambda.\notag
 \end{align*}
Using the weighted DSW lemma, we obtain
\begin{align*}
\degg_X&(Y-\BCH(X,Y))_{\mathrm{odd}}=&\\
&=\mathfrak D^{\mathrm L}\biggl(-\int_{\lambda=0}^1\sum_{k=0}^\infty\biggl(
\lambda^k(\lambda-1)^k\mathcal R^{(\lambda)}(\exp X)(\mathcal R^{(\lambda)}(\exp Y)\mathcal R^{(\lambda)}(\exp X))^{k}\biggr) \\
&\qquad+\sum_{k=0}^\infty
\biggl( \notag
 \lambda^{k}(\lambda-1)^{k}(\lambda-\tfrac12)(\mathcal R^{(\lambda)}(\exp Y)\mathcal R^{(\lambda)}(\exp X))^{k+1} \biggr)
\mathrm d\lambda\biggr)\\
&=-\mathfrak D^{\mathrm L}\int_{\lambda=0}^1\left(\frac12\cdot\frac{\exp(Y)+1}{\exp(Y)-1}\sum_{k=0}^\infty
\lambda^{k}(\lambda-1)^{k } (\mathcal R^{(\lambda)}(\exp Y)\mathcal R^{(\lambda)}(\exp X))^{k+1}
\right)\mathrm d\lambda.
\end{align*}
Therefore we see that
\[\degg_X \eta(\ad Y)B_X(X,Y)_{\mathrm{odd}}=\frac12\degg_X (Y-\BCH(X,Y))_{\mathrm{odd}}.\]
As $\degg_X$ is sufficiently informative in this situation, we obtain the statement.
\end{proof}
Identities \eqref{eq:AY} and \eqref{eq:BY} can be treated in analogous manner.
This finishes the proof of Lemma \ref{lem:oddkv} and therefore of Theorem \ref{th:oddkv}.\qed

Then
\[A_{\mathrm{symm}}(X,Y)_{\mathrm{odd}}=\frac{A_X(X,Y)_{\mathrm{odd}}+ A_Y(X,Y)_{\mathrm{odd}} }{2}\]
and
\[B_{\mathrm{symm}}(X,Y)_{\mathrm{odd}}=\frac{B_X(X,Y)_{\mathrm{odd}}+ B_Y(X,Y)_{\mathrm{odd}} }{2}\]
provide a nice symmetric solution for the odd part of the Kashiwara--Vergne problem.
This symmetric solution, however, does not come from Burgunder's splitting of $X+Y-\BCH(Y,X)$.
(In general, $\delta^{\mathrm L}C(X,Y)\neq \frac12\delta^{\mathrm L}_XC(X,Y)+\frac12\delta^{\mathrm L}_YC(X,Y)$,
but we have only just some convex combinations bidegree-wise.)

In the even part of the KV problem, such simplistic constructions do not seem to work.
At the present, all solutions use advanced methods.
Alexeev, Meinrenken \cite{AM2} uses the star product of Kontsevich, see Kontsevich \cite{K}, Torossian \cite{Tor};
and
Alexeev, Torossian \cite{AT} uses Drinfeld's Knizhnik--Zamolodchikov associator, see
Drinfeld \cite{Dr}, Alexeev, Enriquez, Torossian \cite{AT}.
Actually, the situation is quite nice in low orders $n<10$, but it is
dominated by ambiguities in higher orders, see  Alexeev, Torossian \cite{AT},
Albert, Harinck, Torossian \cite{AHT},
and one has troubles in picking up easily presentable solutions (in even orders).
Some partial information is known regarding the solutions if $X$ or $Y$ have very low degrees, see Alexeev, Petracci \cite{AP}.
\snewpage
\begin{commentx}
\begin{lemma}\plabel{lem:onesided}
(a)
\[\Starting_X\mathcal R^{(\lambda)}((\exp X)(\exp Y) )=(\exp(X)-1)\exp(Y)\bigl(1+(\lambda-1)\mathcal R^{(\lambda)}((\exp X)(\exp Y) )  \bigr)\]
\[=(\exp(X)-1)\bigl(1+(\lambda-1)\mathcal R^{(\lambda)}((\exp Y)(\exp X) ) \bigr)\exp(Y)\]
\[=1-(\exp Y)+(\lambda+(1-\lambda)\exp(Y))\mathcal R^{(\lambda)}((\exp X)(\exp Y) )
\]
\[=\left(1+(\lambda-1)\mathcal R^{(\lambda)}(\exp Y)  \right)^{-1}\left(\mathcal R^{(\lambda)}((\exp X)(\exp Y))-\mathcal R^{(\lambda)}(\exp Y)\right) \]

(b)
\[\Starting_Y\mathcal R^{(\lambda)}((\exp X)(\exp Y) )=(\exp(Y)-1)       \bigl(1+(\lambda-1)\mathcal R^{(\lambda)}((\exp X)(\exp Y) )  \bigr)\]
\[=\left(1+\lambda\mathcal R^{(\lambda)}(\exp X)  \right)^{-1}
\left(\mathcal R^{(\lambda)}((\exp X)(\exp Y))-\mathcal R^{(\lambda)}(\exp X)\right) \]

(c)
\[\Ending_X\mathcal R^{(\lambda)}((\exp X)(\exp Y) )= \bigl(1+(\lambda-1)\mathcal R^{(\lambda)}((\exp X)(\exp Y) )  \bigr)(\exp(X)-1)\]
\[=
\left(\mathcal R^{(\lambda)}((\exp X)(\exp Y))-\mathcal R^{(\lambda)}(\exp Y)\right)
\left(1+\lambda\mathcal R^{(\lambda)}(\exp Y)  \right)^{-1}
\]

(d)
\[\Ending_Y\mathcal R^{(\lambda)}((\exp X)(\exp Y) )= \bigl(1+(\lambda-1)\mathcal R^{(\lambda)}((\exp X)(\exp Y) )  \bigr)\exp(X)(\exp(Y)-1)\]
\[=1-(\exp X)+\mathcal R^{(\lambda)}((\exp X)(\exp Y) )(\lambda+(1-\lambda)\exp X )\]
\[=
\left(\mathcal R^{(\lambda)}((\exp X)(\exp Y))-\mathcal R^{(\lambda)}(\exp X)\right)
\left(1+(\lambda-1)\mathcal R^{(\lambda)}(\exp X)  \right)^{-1}
\]
\end{lemma}
\begin{lemma}\plabel{lem:twosided}

(a)
\begin{multline*}
\Starting_X\Ending_X\mathcal R^{(\lambda)}((\exp X)(\exp Y) )=(\exp(X)-1)+\\+(\exp(X)-1)\exp(Y)(\lambda-1)
\bigl(1+(\lambda-1)\mathcal R^{(\lambda)}((\exp X)(\exp Y) )  \bigr)(\exp(X)-1)
\end{multline*}

(b)
\begin{multline*}
\Starting_X\Ending_Y\mathcal R^{(\lambda)}((\exp X)(\exp Y) )=(\exp(X)-1)(\exp(Y)-1)\\
+(\exp(X)-1)\exp(Y)(\lambda-1)\bigl(1+(\lambda-1)\mathcal R^{(\lambda)}((\exp X)(\exp Y) )  \bigr)\exp(X)(\exp(Y)-1)
\\
=(\exp(X)-1)\lambda\bigl(1+(\lambda-1)\mathcal R^{(\lambda)}((\exp Y)(\exp X) )  \bigr)(\exp(Y)-1)
\end{multline*}

(c)
\begin{multline*}
\Starting_Y\Ending_X\mathcal R^{(\lambda)}((\exp X)(\exp Y) )=\\=
 (\exp(Y)-1) (\lambda-1)  \bigl(1+(\lambda-1)\mathcal R^{(\lambda)}((\exp X)(\exp Y) )  \bigr)(\exp(X)-1)
\end{multline*}

(d)
\begin{multline*}
\Starting_Y\Ending_Y\mathcal R^{(\lambda)}((\exp X)(\exp Y) )=(\exp(Y)-1)  +\\+  (\exp(Y)-1) (\lambda-1)
\bigl(1+(\lambda-1)\mathcal R^{(\lambda)}((\exp X)(\exp Y) )  \bigr)\exp(X)(\exp(Y)-1)
 \end{multline*}

\end{lemma}
\end{commentx}
\snewpage

\end{document}